\documentclass[article]{amsart}

\usepackage{amssymb,amsfonts,amsmath,amsthm}
\usepackage[all,arc]{xy}
\usepackage{enumerate}
\usepackage{mathrsfs}
\usepackage[toc,page]{appendix}
\usepackage[left=3cm, right=3cm, bottom=3cm]{geometry}
\usepackage{tabularx}
\usepackage{url}
\usepackage{color}
\usepackage{tikz-cd}  
\usepackage{mathdots} 
\usepackage{romannum} 
\usepackage{hyperref} 

\newtheorem{thm}{Theorem}[section]

\newtheorem*{thm*}{Theorem}
\newtheorem*{cor*}{Corollary}
\newtheorem*{prop*}{Proposition}
\newtheorem{cor}[thm]{Corollary}
\newtheorem{prop}[thm]{Proposition}
\newtheorem{lem}[thm]{Lemma}

\theoremstyle{definition}
\newtheorem{defn}[thm]{Definition}

\newtheorem*{notn*}{Notation}

\theoremstyle{remark}

\newtheorem*{idea*}{Idea}

\newcommand{\Spec}{{\rm Spec}}

\newcommand{\residual}{\rm Res}

\newcommand{\id}{\rm id}

\newcommand{\parahwt}{\mathcal{P}_{\theta}}
\newcommand{\disc}{\mathcal{O}}

\newcommand{\lieparahwt}{\mathfrak{p}_{\theta}}

\newcommand{\ppower}{\rm AS}

\makeatletter
\let\c@equation\c@thm
\makeatother
\numberwithin{thm}{section}
\numberwithin{equation}{section}

\title[Tame Parahoric Nonabelian Hodge Correspondence]{Tame Parahoric Nonabelian Hodge Correspondence in Positive Characteristic over Algebraic Curves }

\author{Mao Li and Hao Sun}

\begin{document}
\pagenumbering{arabic}
\maketitle
\begin{abstract}
Let $G$ be a reductive group, and let $X$ be an algebraic curve over an algebraically closed field $k$ with positive characteristic. We prove a version of nonabelian Hodge correspondence for tame $G$-local systems over $X$ and logarithmic $G$-Higgs bundles over the Frobenius twist $X'$. To obtain a full description of the correspondence for the noncompact case, we introduce the language of parahoric group schemes to establish the correspondence.
\end{abstract}

\flushbottom



\renewcommand{\thefootnote}{\fnsymbol{footnote}}
\footnotetext[2]{MSC2020 Class: 14D20, 14C30, 20G15, 14L15}

\section{Introduction}

\subsection{Background}
Let $X$ be a smooth projective variety. When the characteristic is zero, Simpson gave a correspondence between Higgs bundles and local systems on $X$ \cite{simpson}. This correspondence is analytic in nature but it does not preserve the algebraic structure. In positive characteristic, the work of Ogus--Vologodsky and Chen--Zhu show that the connection between Higgs bundles and local systems is much closer \cite{Nonabelian Hodge prime char, Nonabelian Hodge prime char general}. Especially, when $X$ is an algebraic curve, Chen--Zhu shows that the stack of $G$-local systems on $X$ is a twisted version of the stack of $G$-Higgs bundles on $X'$, which is the Frobenius twist of $X$. This result is a crucial ingredient in the proof of the geometric Langlands conjecture in positive characteristic on curves \cite{Geometric langlands in prime char}.

For the noncompact case, Simpson established such a correspondence under ``tameness" condition in characteristic zero, and he introduced filtered (parabolic) Higgs bundles and filtered local systems to give a precise description of this correspondence \cite{Simp}. Inspired by Simpson's work, people first introduced parabolic $G$-Higgs bundles and parabolic $G$-local systems to establish this correspondence for principal bundles \cite{BGM}. However, the parabolic objects are not enough to establishing the correspondence completely. With a careful discussion of the local data, Boalch introduced the language of \emph{parahoric subgroups} and give a precise description of the correspondence locally \cite[\S 6]{Riemann Hilbert}. With the help of the local study, it is believed that the language of parahoric group is the correct one to give a full description of the nonabelian Hodge correspondence on noncompact curves for $G$-local systems and $G$-Higgs bundles \cite{BGM,HKSZ22,HS23,Parahoric Higgs}.

In positive characteristic, one may ask a similar question: \emph{how to construct the nonabelian Hodge correspondence on noncompact curves under ``tameness" condition?} In this paper, we establish a version of this correspondence: \emph{tame parahoric nonabelian Hodge correspondence in positive characteristic over curves}. The word ``tameness" has two stories here:
\begin{enumerate}
\item First, \emph{tameness} is a condition on Higgs fields introduced in \cite{Simp}, which can be regarded as \emph{regular singularities} on connections or \emph{logarithmic case} of Higgs bundles when extended to the compactification. The \emph{tameness} condition is a crucial property to establish the correspondence.

\item Second, since we work in positive characteristic, it is related to the concept of \emph{tamely ramified coverings}. In characteristic zero, Balaji--Seshadri found that parahoric torsors over $X$ correspond to $\Gamma$-equivariant bundles over $Y$, where $Y \rightarrow X$ is a Galois covering with Galois group $\Gamma$ \cite{Moduli of parahoric torsors}. When generalizing this approach to positive characteristic $p$, we suppose that the order of $\Gamma$ and the characteristic $p$ are coprime, which means that the covering $Y \rightarrow X$ is tamely ramified and the stack $[Y/\Gamma]$ is tame \cite{Tame Stack}.
\end{enumerate}
This problem has been investigated in \cite{tame for GLn} for local systems with nilpotent residues. Compared with this work, we do not impose any condition on the residue of logarithmic connections. Moreover, some of the investigations are motivated by the results in \cite{Riemann Hilbert}, especially in the local case.

\subsection{Main Results}
\subsubsection{Local Case}
Let $G$ be a connected reductive linear algebraic group with maximal torus $T$ over an algebraically closed field $k$. Let $\mathcal{O}:=k[[z]]$ and $\mathcal{K}:=k((z))$. A parahoric (sub)group is usually understood as a subgroup of $G(\mathcal{K})$ (see \cite{BT2,Twisted loop groups}). In characteristic zero, given a weight (rational cocharacter) $\theta$ of $T$, Boalch constructed a special parahoric group $\mathcal{P}_{\theta}(\mathcal{O}) \subseteq G(\mathcal{K})$ to study the tame Riemann-Hilbert correspondence locally \cite{Riemann Hilbert}. As a special case, when $\theta=0$, then $\mathcal{P}_\theta(\mathcal{O})$ is exactly $G(\mathcal{O})$, which goes back to the case of $G$-bundles.

We first study this problem locally on a formal disc $\mathbb{D}:={\rm Spec}(\mathcal{O})$ with $\theta=0$. In this case, a logahoric $\mathcal{P}_\theta(\mathcal{O})$-connection is exactly a logarithmic $G$-connection (or a \emph{tame $G$-local system}) over $\mathbb{D}$, i.e.
\begin{align*}
d - A\frac{dz}{z}, \quad A=\sum_{i \geq 0}  a_i z^i \in \mathfrak{g}(\mathcal{O}),
\end{align*}
and a logahoric $\mathcal{P}_\theta(\mathcal{O})$-Higgs field is a logarithmic $G$-Higgs field over $\mathbb{D}$. The word ``logahoric" is a blend of logarithmic and parahoric introduced in \cite{Riemann Hilbert}. Fixing a splitting of the Lie algebra $\pi: \mathfrak{t}_{\mathbb{F}_p} \hookrightarrow \mathfrak{t}$, we can decompose $a_0$ as
\begin{align*}
a_0=s+n=\tau+\sigma+n,
\end{align*}
where $s$ (resp. $n$) is the semisimple part (resp. nilpotent part) of $a_0$, and $\tau \in \mathfrak{t}_{\mathbb{F}_p}$ is called the \emph{rational} part of $s$, while $\sigma$ is called the \emph{irrational} part of $s$. By generalizing a classical calculation to positive characteristic, a logarithmic $G$-connection is equivalent to one in the form $d - B\frac{dz}{z}$  under the gauge action, where $B=\sum_{i \geq 0}b_i z^i$ such that $b_i$ lies in the generalized eigenspace of the operator $[b_0,-]$ with eigenvalue $i$ (see Lemma~\ref{standard form 1}). Therefore, we can assume that given a logarithmic $G$-connection $d - A\frac{dz}{z}$, the term $a_i$ lies in the generalized eigenspace of the operator tor $[\tau,-]$ with eigenvalue $i$.

We first consider the case that the semisimple part of $a_0$ is irrational, i.e. $\tau=0$.
\begin{itemize}
\item ${\rm Locsys}^{tame}_{G,irr}(\mathbb{D})$ is the category of tame $G$-local systems $(E,\nabla)$ on the formal disc $\mathbb{D}$ such that the semisimple part of the residue of $\nabla$ is irrational (under gauge action).
\item ${\rm Higgs}^{tame}_{G,irr}(\mathbb{D}')$ is the category of logarithmic $G$-Higgs bundles $(E,\phi)$ on the Frobenius twist of $\mathbb{D}$ such that the semisimple part of the residue of $\phi$ is irrational (under adjoint action).
\end{itemize}
We prove that these two categories are equivalent (Proposition~\ref{standard form in irrational})
\begin{align*}
{\rm Locsys}^{tame}_{G,irr}(\mathbb{D}) \cong {\rm Higgs}^{tame}_{G,irr}(\mathbb{D}').
\end{align*}
If the rational part $\tau$ of $a_0$ is nontrivial, we define the category ${\rm Locsys}^{tame}_{G,\tau}(\mathbb{D})$ similarly. We prove the equivalence
\begin{align*}
    {\rm Locsys}^{tame}_{G,\tau}(\mathbb{D}) \cong {\rm Higgs}^{tame}_{G'_{\tau},irr}(\mathbb{D}'),
\end{align*}
where ${\rm Higgs}^{tame}_{G'_{\tau},irr}(\mathbb{D}')$ is the category of logahoric $G'_{\tau}(\mathcal{O})$-Higgs bundles over $\mathbb{D}'$, where $G'_{\tau}(\mathcal{O})$ is regarded as a parahoric group over $\mathbb{D}'$ (Lemma~\ref{parahoric of rational semisimple}), such that semisimple part of the residue is irrational (Proposition~\ref{standard form in general}). This finishes the discussion of the local case and the correspondence we obtain is called \emph{local tame nonabelian Hodge correspondence}.

Next, we come to the parahoric case over $\mathbb{D}$. We fix a tame weight $\theta$, which is a weight such that the denominator $d$ is coprime to $p$. Similar to the above discussion, we define two categories:
\begin{itemize}
\item ${\rm Locsys}^{tame}_{\mathcal{P}_\theta,\tau}(\mathbb{D})$ is the category of tame $\mathcal{P}_\theta(\mathcal{O})$-local systems on $\mathbb{D}$ such that the semisimple part of the residue is $\tau$ (under gauge action);
\item ${\rm Higgs}^{tame}_{G'_{\theta+\tau},irr}(\mathbb{D}')$ is the category of logahoric $G'_{\theta+\tau}(\mathcal{O})$-Higgs bundles on $\mathbb{D}'$ such that the semisimple part of the residue is irrational (under adjoint action).
\end{itemize}
The following diagram
\begin{center}
\begin{tikzcd}
{\rm Locsys}^{tame}_{\mathcal{P}_\theta,\tau}(\mathbb{D}_x )  \arrow[d,dotted] \arrow[rr,"\text{Proposition~\ref{parahoric connection equivariant}}"]  & & {\rm Locsys}^{tame}_{G, \theta+\tau}([\mathbb{D}_y /\Gamma]) \arrow[d, "\text{Proposition~\ref{standard form in general}}"] \\
{\rm Higgs}^{tame}_{G'_{\theta+\tau},irr}(\mathbb{D}'_x )  & & {\rm Higgs}^{tame}_{G'_{\theta+\tau}, irr}([\mathbb{D}'_y /\Gamma]) \arrow[ll, "\text{Theorem~\ref{equivalence for higgs bundles}}"]
\end{tikzcd}
\end{center}
implies the equivalence of these two categories. Actually, each arrow in the diagram can be understood as an equivalence and the horizontal arrows are discussed in \S\ref{sect corresp} based the the correspondence studied by Balaji--Seshadri \cite{Moduli of parahoric torsors}.
\begin{thm}[Theorem \ref{classification for parahoric connection}]
We have an equivalence of categories
\begin{align*}
{\rm Locsys}^{tame}_{\mathcal{P}_\theta,\tau}(\mathbb{D}) \cong {\rm Higgs}^{tame}_{G'_{\theta+\tau},irr}(\mathbb{D}').
\end{align*}
Moreover, the $p$-curvature of the parahoric $\mathcal{P}_\theta(\mathcal{O})$-connection is zero if and only if the corresponding parahoric $G'_{\theta+\tau}(\mathcal{O})$-Higgs bundle has zero Higgs field.
\end{thm}
\noindent This correspondence is called the \emph{local tame parahoric nonabelian Hodge correspondence}.

\subsubsection{Global Case}
After the local description of tame parahoric nonabelian Hodge correspondence, we move to the global case and we consider the case $\theta=0$ first. Let $X$ be a smooth algebraic curve with a given reduced effective divisor $D$. Let $\mathscr{L}:=\Omega_X(D)$, and denote by $\mathscr{L}'$ the corresponding line bundle over $X'$, which is the Frobenius twist of $X$, and we have $Fr^* \mathscr{L}' \cong \mathscr{L}^p$, where $Fr: X \rightarrow X'$ is the Frobenius morphism. Let $B_{\mathscr{L}'}$ be the $\mathscr{L}'$-twisted Hitchin base. Then we have a natural morphism
\begin{align*}
h_p: {\rm Locsys}^{tame}_G \rightarrow B_{\mathscr{L}'},
\end{align*}
which is called \emph{$p$-Hitchin morphism} (see \cite[Proposition 3.1]{Nonabelian Hodge prime char} or \S\ref{sect global}). The $p$-Hitchin morphism helps us to construct a group schemes $J^p$ over $X \times B_{\mathscr{L}'}$ equipped with a natural action
\begin{align*}
{\rm Locsys}^{tame}_{J^p} \times_{B_{\mathscr{L}'}} {\rm Locsys}^{tame}_{G} \rightarrow {\rm Locsys}^{tame}_{G}.
\end{align*}
Now we define a substack $\mathcal{A} \subseteq {\rm Locsys}^{tame}_{J^p}$, of which the connections are with vanishing $p$-curvature, and let $\mathcal{X}$ be the stack parametrizing $(E,\nabla,\Psi)$ such that
\begin{itemize}
\item $(E,\nabla)$ is a tame $G$-local system with zero $p$-curvature,
\item $\Psi$ is a horizontal section of ${\rm Ad}(E)\otimes Fr^*\mathscr{L}'$
\end{itemize}
Clearly, $\mathcal{X}$ is an algebraic stack over $B_{\mathscr{L}'}$. In the meantime, we can construct a vector bundle $\mathscr{B}_{\mathscr{L}'}$ over $B_{\mathscr{L}'}$, of which the fiber is $H^0(X', {\rm Lie}(J'_{b'}) \otimes \mathscr{L}')$ for each $b' \in B_{\mathscr{L}'}$. We prove that ${\rm Locsys}^{tame}_{J^p}$ is smooth over $\mathscr{B}_{\mathscr{L}'}$ (Lemma~\ref{smoothness of locsys}). Then, we define the stack $\mathcal{H}$
\begin{center}
\begin{tikzcd}
\mathcal{H}  \arrow[d,dotted] \arrow[r,dotted]  & {\rm Locsys}^{tame}_{J^p} \arrow[d] \\
B_{\mathscr{L}'} \arrow[r, "\tau'"] & \mathscr{B}_{\mathscr{L}'}
\end{tikzcd}
\end{center}
as the pullback of the diagram. With the help of the local study in \S\ref{sect local}, we follow Chen--Zhu's approach to prove the main results in this paper.
\begin{thm}[Theorem~\ref{first structure theorem for G}]
There exists a canonical isomorphism of stacks over $B_{\mathscr{L}'}$:
\begin{align*}
\mathcal{H}\times^{\mathcal{A}}\mathcal{X}\rightarrow {\rm Locsys}^{tame}_{G}.
\end{align*}
\end{thm}
Moreover, the structure of $\mathcal{X}$ can be described by the rational semisimple residues $\boldsymbol\tau=\{\tau_x| x \in D\}$.

\begin{prop}[Proposition~\ref{structure of X}]
The stack $\mathcal{X}$ is the disjoint union of $\mathcal{X}_{\boldsymbol\tau}$, where the union ranges over the all rational semisimple conjugacy classes of $\boldsymbol\tau$, and we have
\begin{align*}
\mathcal{X}_{\boldsymbol\tau} \cong {\rm Higgs}^{tame}_{G'_{\boldsymbol\tau}}.
\end{align*}
\end{prop}

For the parahoric case, we fix a collection of tame weights $\boldsymbol\theta=\{\theta_x, x \in D\}$.  With respect to the data $(X,D,\boldsymbol\theta)$, we can construct two objects:
\begin{itemize}
\item a parahoric group scheme $\mathcal{P}_{\boldsymbol\theta}$ over $X$;
\item a tamely ramified covering $Y \rightarrow X$.
\end{itemize}
Balaji--Seshadri showed that parahoric torsors are equivalent to $\Gamma$-equivariant $G$-bundles \cite{Moduli of parahoric torsors} in characteristic zero. Their approaches can be generalized to Higgs bundles and local systems in mixed characteristic (see \S\ref{sect corresp} or \cite{Parahoric Higgs}). We show the following isomorphisms of stacks (Theorem~\ref{equivalence for higgs bundles} and Proposition~\ref{parahoric connection equivariant})
\begin{align*}
{\rm Higgs}^{tame,{\boldsymbol\rho}}_{G}([Y/\Gamma]) \cong {\rm Higgs}^{tame}_{\mathcal{P}_{\boldsymbol\theta}}(X), \quad {\rm Locsys}^{tame, {\boldsymbol\rho}}_G([Y/\Gamma]) \cong {\rm Locsys}^{tame}_{\mathcal{P}_{\boldsymbol\theta}}(X),
\end{align*}
where $\boldsymbol\rho$ is a topological data corresponding to weights $\boldsymbol\theta$ and $[Y/\Gamma]$ is the quotient stack, of which the coarse moduli space is $X$. This equivalence shows that proving the tame parahoric nonabelian Hodge correposndence on $X$ is equivalent to considering the correspondence over $[Y/\Gamma]$ (or an $\Gamma$-equivariant version over $Y$), which is a direct result of Theorem \ref{first structure theorem for G}.

\begin{thm}[Theorem~\ref{parahoric non hod corr} and Proposition~\ref{structure of parahoric X}]
There exists a canonical isomorphism of stacks over $B_{\mathscr{L}'}$:
\begin{align*}
\mathcal{H}\times^{\mathcal{A}}\mathcal{X}_{\mathcal{P}_{\boldsymbol\theta}}\rightarrow {\rm Locsys}^{tame}_{\mathcal{P}_{\boldsymbol\theta}}.
\end{align*}
The stack $\mathcal{X}_{\mathcal{P}_{\boldsymbol\theta}}$ is the disjoint union of $\mathcal{X}_{\mathcal{P}_{\boldsymbol\theta}, {\boldsymbol\tau}}$ (over $\boldsymbol\tau$) such that
\begin{align*}
\mathcal{X}_{\mathcal{P}_{\boldsymbol\theta}, {\boldsymbol\tau}} \cong {\rm Higgs}^{tame}_{G'_{{\boldsymbol\theta}+{\boldsymbol\tau}}}(X').
\end{align*}
\end{thm}
\noindent This correspondence is called the \emph{tame parahoric nonabelian Hodge correspondence}.

\subsection{Structure of the Paper}
In \S\ref{sect pre}, we briefly review some necessary backgrounds about parahoric groups. In \S\ref{sect corresp}, we generalize the correspondence studied by Balaji--Seshadri \cite{parabolic torsors} to positive characteristics. More precisely, let $Y \rightarrow X$ be a covering with Galois group $\Gamma$, and we prove that there is a correspondence between parahoric Higgs bundles (resp. local systems) over $X$ and $\Gamma$-equivariant Higgs bundles (resp. local systems) over $Y$. In \S\ref{sect local}, we construct the tame parahoric nonabelian Hodge correspondence on a formal disc $\mathbb{D}$ (Theorem~\ref{classification for parahoric connection}). Based on the local study, we prove the \emph{tame parahoric nonabelian Hodge correspondence} (Theorem~\ref{first structure theorem for G} and \ref{parahoric non hod corr}).

\vspace{2mm}
\textbf{Acknowledgments}.
The authors would like to thank D. Arinkin, G. Kydonakis, and L. Zhao for helpful discussions and conversations. During the preparation of the paper, Hao Sun is partially supported by National Key R$\&$D Program of China (No. 2022YFA1006600).
\vspace{2mm}

\section{Preliminaries}\label{sect pre}
Let $k$ be an algebraically closed field of positive characteristic. Let $G$ be a connected reductive linear algebraic group over $k$. We fix a maximal torus $T \subseteq G$. Denote by $\mathfrak{g}$ and $\mathfrak{t}$ the Lie algebras of $G$ and $T$ respectively. Let $X^*(T):={\rm Hom}(T,\mathbb{G}_m)$ be the group of characters, and let $X_*(T):={\rm Hom}(\mathbb{G}_m,T)$ be the group of cocharacters. The adjoint action of $T$ on $\mathfrak{g}$ gives a decomposition of the Lie algebra
\begin{align*}
\mathfrak{g}=\bigoplus\limits_{\alpha \in X^*(T)} \mathfrak{g}_\alpha.
\end{align*}
A \emph{root} is a nonzero character $\alpha$, of which $\mathfrak{g}_\alpha \neq 0$, and $\mathfrak{g}_\alpha$ is called a \emph{root space}. Denote by $R$ the set of roots. We define a natural pairing
\begin{align*}
\langle \cdot, \cdot \rangle : X_*(T) \times X^*(T) \rightarrow \mathbb{Z}
\end{align*}
This pairing can be extended to $\mathbb{Q}$ naturally. Thus, $\langle \theta,\alpha \rangle$ is a well-defined rational number, where $\theta$ is a weight and $\alpha$ is a root. We also use the notation $\alpha(\theta)$ for this number.

\subsection*{\textbf{Tame Weights}}

\hfill{\space}

A rational cocharacter $\theta \in X_*(T) \otimes_{\mathbb{Z}} \mathbb{Q}$ is called a \emph{weight}. In characteristic zero, a weight $\theta$ can be regarded as an element in $\mathfrak{t}_\mathbb{Q}$ under differentiation. In positive characteristic, we need some extra conditions.
\begin{defn}\label{defn tame weight}
A weight $\theta$ is \emph{tame} if its denominator is coprime to $p$.
\end{defn}
Given a tame weight $\theta= \vartheta \otimes \frac{a}{b}$, we have $(b,p)=1$. Then $\theta$ corresponds to a well-defined element in $\mathfrak{t}_{\mathbb{F}_p}$. Abusing the notation, a tame weight can either be a rational cocharacter or an element in $\mathfrak{t}_{\mathbb{F}_p}$. Now given $a \in \mathfrak{g}$, let $a=s+n$ be its Jordan decomposition, where $s$ is the semisimple part and $n$ is the nilpotent part. Denote by $\mathfrak{t} \subseteq \mathfrak{g}$ the corresponding Lie algebras. Now we fix a splitting of Lie algebras $\pi: \mathfrak{t}_{\mathbb{F}_p} \hookrightarrow \mathfrak{t}$ and decompose $s=\tau+\sigma$ such that under an appropriate conjugation, we have $\tau \in \mathfrak{t}_{\mathbb{F}_p}$ and the projection of $\sigma$ onto $\mathfrak{t}_{\mathbb{F}_p}$ is zero. We say that $\tau$ is the \emph{rational part} of $s$ and $\sigma$ is the \emph{irrational part} of $s$. For convenience, we sometimes assume that $\tau \in \mathfrak{t}_{\mathbb{F}_p}$, and then we have $\alpha(\tau) \in \mathbb{F}_p$, while $\alpha(\sigma) \notin \mathbb{F}_p$ for every root $\alpha$.

\subsection*{\textbf{Grading of Lie algebras}}

\hfill{\space}

\indent Given a weight $\theta$, it also induces a (graded) decomposition of the Lie algebra $\mathfrak{g}$
\begin{align*}
\mathfrak{g}=\bigoplus\limits_{\lambda \in \mathbb{Q}} \mathfrak{g}_\lambda,
\end{align*}
where $\mathfrak{g}_\lambda=\bigoplus\limits_{\alpha(\theta)=\lambda} \mathfrak{g}_\alpha$. Note that this decomposition is not given by the eigenspace of the differentiation of the adjoint action ${\rm Ad}(\theta)$, of which the eigenvalues are in $\mathbb{F}_p$. For any integer $k$, we define
\begin{align*}
\mathfrak{g}_{\geq k}:= \bigoplus_{\lambda \geq -k} \mathfrak{g}_\lambda \subseteq \mathfrak{g}.
\end{align*}
As a special case, $\mathfrak{g}_{\geq 0} = \mathfrak{p}_\theta$ is the Lie algebra of the parabolic subgroup $P_\theta \subseteq G$ associated to $\theta$, and $\mathfrak{g}_{0} = \mathfrak{l}_\theta$ is the Lie algebra of the Levi component $L_\theta$ of $P_\theta$.

\subsection*{\textbf{Parahoric groups}}

\hfill{\space}

Given a root $\alpha \in R$, there is a natural isomorphism
\begin{align*}
{\rm Lie}(\mathbb{G}_a) \rightarrow \mathfrak{g}_\alpha.
\end{align*}
This isomorphism induces a natural homomorphism
\begin{align*}
u_\alpha: \mathbb{G}_a \rightarrow G,
\end{align*}
such that $t u_\alpha(a) t^{-1}= u_\alpha(\alpha(t)a)$ for $t \in T$ and $a \in \mathbb{G}_a$. Denote by $U_\alpha$ the image of $u_\alpha$, which is a closed subgroup. Then, we define $\mathcal{O}:=k[[z]]$ and $\mathcal{K}:=k((z))$. Denote by $LG:=G(\mathcal{K})$ the loop group and $L\mathfrak{g}:=\mathfrak{g}(\mathcal{K})$ the loop Lie algebra (see \cite{Twisted loop groups} for more details).
\begin{defn}\label{defn local parah gp}
Given a weight $\theta$, the \emph{parahoric (sub)group} $\mathcal{P}_\theta(\mathcal{O}) \subseteq LG$ is defined as
\begin{align*}
\mathcal{P}_{\theta}(\mathcal{O}):=\langle T(\mathcal{O}), \, U_\alpha (z^{m_\alpha(\theta)}\mathcal{O}), \, \alpha \in R \rangle,
\end{align*}
where $m_\alpha(\theta):=\lceil -\alpha(\theta) \rceil$ and $\lceil \ \cdot \ \rceil$ is the ceiling function. The \emph{Levi subgroup $\mathcal{L}_\theta(\mathcal{O})$} of $\mathcal{P}_\theta(\mathcal{O})$ is defined as
\begin{align*}
\mathcal{L}_\theta(\mathcal{O}):=\langle T(\mathcal{O}), \, U_\alpha(\mathcal{O}), \, \alpha \in R \text{ and }\alpha(\theta)=0 \rangle.
\end{align*}
\end{defn}
\noindent As a special case, when $\theta=0$, then $\mathcal{P}_{\theta}(\mathcal{O})=G(\mathcal{O})$. Fixing a weight $\theta$, we define a grading of $L\mathfrak{g}$
\begin{align*}
L \mathfrak{g}_{\geq k}:= \{\sum_{i \in \mathbb{Z}} a_i z^i \in L \mathfrak{g} \text{ }| \text{ } a_i \in \mathfrak{g}_{\geq i-k} \},
\end{align*}
where $k \in \mathbb{Q}$. If $k \geq 0$, denote by $\mathcal{G}_{\geq k}$ the associated group in $LG$. Here are some examples. When $k=0$, $\mathfrak{p}_\theta(\mathcal{O}):=L \mathfrak{g}_{\geq 0}$ is the Lie algebra of the parahoric group $\mathcal{P}_\theta(\mathcal{O})$. If $\theta=0$, the filtration $L\mathfrak{g}_{\geq k}$ is the natural filtration of $L\mathfrak{g}$ based on the degree.

\subsection*{\textbf{Parahoric group schemes}}

\hfill{\space}

The definition of parahoric groups is a local picture of parahoric group schemes. Let $X$ be a smooth algebraic curve over $k$, and we also fix a reduced effective divisor $D$ on $X$, which is a set of $s$ distinct points. For each point $x \in D$, we equip it with a weight $\theta_x \in Y(T) \otimes_{\mathbb{Z}} \mathbb{Q}$. Let ${\boldsymbol\theta}:=\{\theta_x, \, x \in D\}$ be the collection of weights over points in $D$.

\begin{defn}\label{defn global parah gp}
We define a group scheme $\mathcal{P}_{\boldsymbol\theta}$ over $X$ by gluing the following local data
\begin{align*}
\mathcal{P}_{\boldsymbol\theta}|_{X\backslash D} \cong G \times (X\backslash D), \, \mathcal{P}_{\boldsymbol\theta}|_{\mathbb{D}_x} \cong \mathcal{P}_{\theta_x}(\mathcal{O}), \, x \in D,
\end{align*}
where $\mathbb{D}_x$ is a formal disc around $x$. This group scheme $\mathcal{P}_{\boldsymbol\theta}$ is called a \emph{parahoric group scheme}.
\end{defn}
\noindent By \cite[Lemma 3.18]{ChGP}, the group scheme $\mathcal{P}_{\boldsymbol\theta}$ defined above is a smooth affine group scheme of finite type, flat over $X$. By definition, we know $\mathcal{P}_{\boldsymbol\theta}|_{X \backslash D}\simeq G\times (X \backslash D)$, and then there is a natural $D_{X}$-scheme structure on $G\times X$. Let $\mathscr{L}:=\Omega_X(D)$. These facts induce a natural connection $\nabla_{\mathcal{P}_{\boldsymbol\theta}}$ with first order pole on $\mathcal{O}_{\mathcal{P}_{\boldsymbol\theta}}$:
\begin{align*}
\mathcal{O}_{\mathcal{P}_{\boldsymbol\theta}} \xrightarrow{\nabla_{\mathcal{P}_{\boldsymbol\theta}} } \mathcal{O}_{\mathcal{P}_{\boldsymbol\theta}}\otimes_{\mathcal{O}_{X}}\mathscr{L}
\end{align*}
such that $\nabla_{\mathcal{P}_{\boldsymbol\theta}}$ satisfies the condition $\nabla_{\mathcal{P}_{\boldsymbol\theta}}(fg)=g\otimes df+f\nabla_{\mathcal{P}_{\boldsymbol\theta}}(g)$ for $f$ is a local section of $\mathcal{O}_{X}$ and $g$ is a local section of $\mathcal{O}_{\mathcal{P}_{\boldsymbol\theta}}$.

\subsection*{\textbf{ Tame Parahoric Local systems and Logahoric Higgs bundles}}

\hfill{\space}

Let $E$ be a $\mathcal{P}_{\boldsymbol\theta}$-torsor over $X$, which can be understood by gluing the local data: $G$-torsor $E|_{X \backslash D}$ and $\mathcal{P}_{\theta_x}$-torsor $E|_{\mathbb{D}_x}$. A \emph{logahoric $\mathcal{P}_{\boldsymbol\theta}$-connection} is a connection with first order pole on $E$. More precisely, it is a derivation
\begin{align*}
\nabla: \mathcal{O}_{E} \rightarrow \mathcal{O}_{E} \otimes_{\mathcal{O}_{X}}\mathscr{L}
\end{align*}
such that the following diagram commutes:
$$
\xymatrix{
\mathcal{O}_{E}\ar[r] \ar[d]^{\nabla_{E}} & \mathcal{O}_{\mathcal{P}_{\boldsymbol\theta}} \otimes_{\mathcal{O}_{X}} \mathcal{O}_{E} \ar[d]^{\nabla_{\mathcal{P}_{\boldsymbol\theta}} \otimes\id+\id\otimes\nabla_{\mathcal{P}_{\boldsymbol\theta}}}\\
\mathcal{O}_{E}\otimes_{\mathcal{O}_{X}}\mathscr{L} \ar[r] & O_{\mathcal{P}_{\boldsymbol\theta}}\otimes_{\mathcal{O}_{X}}\mathcal{O}_{E}\otimes_{\mathcal{O}_{X}}\mathscr{L} .
}
$$
\begin{defn}\label{defn Localsys}
A \emph{tame $\mathcal{P}_{\boldsymbol\theta}$-local system} over $X$ is a pair $(E,\nabla)$, where $E$ is a $\mathcal{P}_{\boldsymbol\theta}$-torsor and $\nabla$ is a logahoric $\mathcal{P}_{\boldsymbol\theta}$-connection on $E$. Denote by ${\rm Locsys}^{tame}_{\mathcal{P}_{\boldsymbol\theta}}(X)$ the stack of tame $\mathcal{P}_{\boldsymbol\theta}$-local systems over $X$. If there is no ambiguity, we omit $X$ in the notation ${\rm Locsys}^{tame}_{\mathcal{P}_{\boldsymbol\theta}}(X)$.
\end{defn}

Now let ${\rm Ad}(E)$ be the adjoint bundle of $E$. A section $\phi \in H^0(X, {\rm Ad}(E) \otimes_{\mathcal{O}_X} \mathscr{L})$ is called a \emph{logarithmic $\mathcal{P}_{\boldsymbol\theta}$-Higgs field}.
\begin{defn}\label{defn Higgs bundle}
A \emph{logarhoic $\mathcal{P}_{\boldsymbol\theta}$-Higgs bundle} is a pair $(E,\phi)$, where $E$ is a $\mathcal{P}_{\boldsymbol\theta}$-torsor and $\phi \in H^0(X, {\rm Ad}(E) \otimes \mathscr{L})$ is a logarithmic $\mathcal{P}_{\boldsymbol\theta}$-Higgs field. Denote by ${\rm Higgs}^{tame}_{\mathcal{P}_{\boldsymbol\theta}}(X)$ the stack of logahoric $\mathcal{P}_{\boldsymbol\theta}$-Higgs bundles over $X$. Sometimes, we add the subscript $\mathscr{L}$ and use the notation ${\rm Higgs}^{tame}_{\mathcal{P}_{\boldsymbol\theta},\mathscr{L}}(X)$ to emphasize that it is for $\mathscr{L}$-twisted $\mathcal{P}_{\boldsymbol\theta}$-Higgs bundles.
\end{defn}

\section{Correspondence: Parahoric vs. Equivariant}\label{sect corresp}
Balaji and Seshadri gives the correspondence between parahoric torsors and equivariant bundles \cite{Moduli of parahoric torsors}. A similar correspondence also works for Higgs bundles and local systems \cite{parahoric connections, Parahoric Higgs}. Although the correspondence is only given in characteristic zero, it can be naturally generalized to prime characteristic under some necessary conditions (tame weights). In this section, we first give the correspondence of parahoric torsors and equivariant bundles in positive characteristic, which is a direct generalization of Balaji and Seshadri's work, and then we give the correspondences for Higgs bundles and local systems.

We introduce some notations first. Let $X$ and $Y$ be two smooth algebraic curves over $k$. In fact, we assume that $Y \rightarrow X$ is a covering with Galois group $\Gamma$. Let $x$ (resp. $y$) be a point in $X$ (resp. $Y$). Denote by $\mathbb{D}_x:=\Spec(\mathcal{O}_z)$ (resp. $\mathbb{D}_y:=\Spec(\mathcal{O}_w)$) the formal disc around $x$ (resp. $y$) with local coordinate $z$ (resp. $w$).

\subsection{Torsors}\label{subsect torsor}
Let $\Gamma$ be a cyclic group of order $d$ with a generator $\gamma$. In this paper, we always assume that the order $d$ and the characteristic $p$ are coprime. Choose a $d$-th root of unity $\zeta$, and we have a natural $\Gamma$-action on $\mathbb{D}_y$ such that $\gamma w=\zeta w$.

\begin{defn}
A \emph{$\Gamma$-equivariant $G$-torsor} is a $G$-torsor $F$ over $\mathbb{D}_y$ together with a lift of the action of $\Gamma$ on $F$ preserving the $G$-action. For simplicity, a $\Gamma$-equivariant $G$-torsor is called a \emph{$(\Gamma,G)$-torsor}.
\end{defn}

Since we work on $\mathbb{D}_y$ now, we may assume that $F$ is the trivial $G$-torsor, and
\begin{align*}
F(\mathbb{D}_y) \cong G(\mathcal{O}_w).
\end{align*}
Therefore, given a $(\Gamma,G)$-torsor $F$, we get a morphism $\rho: \Gamma \rightarrow G$ well-defined up to conjugation. If a $(\Gamma,G)$-torsor $F$ corresponds to $\rho$, then we say that $F$ is of \emph{type $\rho$}. By \cite[Lemma 2.5]{parabolic torsors}, the $\Gamma$-equivariant structure is uniquely determined by the representation $\rho$ and one may assume that the $\Gamma$-action on $\mathbb{D}_{y}\times G$ is of the form \begin{align*}
\gamma\cdot (w, h)=(\gamma w, \rho(\gamma) h\rho^{-1}(\gamma)).
\end{align*}
Since $\Gamma$ is a cyclic group, the representation $\rho$ factors through $T$ under a suitable conjugation. Note that
\begin{align*}
{\rm Hom}(\Gamma, T) \cong {\rm Hom}(X_*(T),X_*(\Gamma_y))={\rm Hom}(X_*(T),\mathbb{Z}/d\mathbb{Z})=X_*(T)/d\cdot X_*(T).
\end{align*}
Then, a representation $\rho: \Gamma \rightarrow T$ corresponds to an element in $X_*(T) \otimes_{\mathbb{Z}}\mathbb{Q}$ with denominator $d$, which is a tame weight. Although this correspondence is not unique in general, a representation $\rho$ corresponds to a unique small tame weight $\theta$, which means that the rational number is smaller than one. On the other hand, given a tame weight $\theta$, it corresponds to a unique representation $\rho: \Gamma \rightarrow T$ given by $\rho(\gamma)=\zeta^{d\cdot\theta}$. Furthermore, we can find an element $\Delta(w):=w^{d \cdot\theta} \in T(\mathcal{K}_w)$, which satisfies
\begin{align*}
\Delta(\gamma w)=\rho(\gamma) \Delta(w).
\end{align*}

Let $F$ be a $(\Gamma,G)$-torsor of type $\rho$, and let $\theta$ be a tame weight corresponding to $\rho$. Denote by ${\rm Aut}_{(\Gamma,G)}(F)$ the automorphism group. An automorphism $\sigma \in {\rm Aut}_{(\Gamma,G)}(F)$ is equivalent to an element in $G(\mathbb{C}[[w]])^{\Gamma}$. Let $\Delta(w)$ be the element satisfying $\Delta(\gamma w)= \rho(w)\Delta(w)$. We define $\varsigma:=\Delta^{-1} \sigma \Delta$. Clearly, we have
\begin{align*}
\varsigma(\gamma w)=\varsigma(w),
\end{align*}
which means that $\varsigma$ is $\Gamma$-invariant. Therefore, it can be descended to an element $G(\mathcal{K}_z)$ by substituting $z=w^d$. Moreover, for each root $\alpha \in R$, we have
\begin{align*}
\varsigma(w)_\alpha=\sigma(w)_\alpha w^{-d \cdot \alpha(\theta)},
\end{align*}
where the subscript $\alpha$ means that the element is in $U_\alpha(\mathcal{O}_w)$. Substituting $z=w^d$, we have
\begin{align*}
\varsigma(z)_\alpha=\sigma(z)_\alpha z^{-\alpha(\theta)}.
\end{align*}
Since $\varsigma(z)_\alpha$ is $\Gamma$-invariant, we have $\varsigma(z)_\alpha \in U_\alpha(z^{m_r(\theta)}\mathcal{O}_z)$ for each $\alpha \in R$. In conclusion, the element $\varsigma(z)$ is in $\mathcal{P}_{\theta}(\mathcal{O}_z)$, and the above discussion implies the following lemma:

\begin{lem}[Theorem 2.3.1 in \cite{Moduli of parahoric torsors}]\label{explicit correspondence bewteen groups}
Let $F$ be a $(\Gamma,G)$-torsor of type $\rho$, and let $\theta$ be a weight corresponding to $\rho$. Then, we have
\begin{align*}
{\rm Aut}_{(\Gamma,G)}(F) \cong \mathcal{P}_{\theta}(\mathcal{O}_z).
\end{align*}
More precisely, the automorphism group $\Delta^{-1} {\rm Aut}_{(\Gamma,G)}(F) \Delta$ can be descended to $\mathbb{D}_x$, which is isomorphic to $\mathcal{P}_{\theta}(\mathcal{O}_z)$.
\end{lem}

This isomorphism gives us the following equivalence of categories.
\begin{lem}
The category of $(\Gamma,G)$-torsors of type $\rho$ over $\mathbb{D}_y$ is equivalent to the category of $\mathcal{P}_\theta(\mathcal{O}_z)$-torsors over $\mathbb{D}_x$.
\end{lem}

Now we consider the correspondence globally. Let $Y \rightarrow X$ be a covering of smooth algebraic curves with Galois group $\Gamma$. Suppose that the order of $\Gamma$ is not divided by $p$. Here is the definition of $\Gamma$-equivariant $G$-torsor over $Y$.
\begin{defn}
A \emph{$(\Gamma,G)$-torsor} over $Y$ is a $G$-torsor $F$ together with a lift $\Gamma$-action on $F$ preserving the action of $G$.
\end{defn}
\noindent Equivalently, a $(\Gamma,G)$-torsor over $Y$ is a $G$-torsor over the quotient stack $[Y/\Gamma]$. Denote by ${\rm Bun}_G([Y/\Gamma])$ the stack of $(\Gamma,G)$-torsors over $Y$.

Given $y \in Y$, let $\Gamma_y$ be the stabilizer group of the point $y$, and we suppose that $\Gamma_y$ is a cyclic group. Denote by $R$ the set of points in $Y$, of which the stabilizer groups are nontrivial. This set $R$ is regarded as the set of ramifications, and then denote by $D \subseteq X$ the branch divisor. Let $F$ be a $(\Gamma,G)$-torsor over $Y$. As was discussed above, the $\Gamma$-action around $y \in R$ is given by a representation $\rho_y: \Gamma_y \rightarrow T$. Denote by $\boldsymbol\rho:=\{\rho_y, y \in R\}$ the collection of representations. Let $\theta_y$ be a corresponding weight of $\rho_y$, and then denote by $\boldsymbol\theta$ the collection of weights, which gives a parahoric group scheme $\mathcal{P}_{\boldsymbol\theta}$ over $X$.

A $(\Gamma,G)$-torsor $F$ over $Y$ can be understood by gluing the following local data. We define $F_y:= \mathbb{D}_y  \times G$, such that the $\Gamma_y$-action is defined as
\begin{align*}
\gamma \cdot (u,g) \rightarrow (\gamma u, \rho_y(\gamma)g), \quad u \in \mathbb{D}_y, \gamma \in \Gamma_y,
\end{align*}
and define $F_0:=(Y \backslash R) \times G$ with the $\Gamma_y$-structure
\begin{align*}
\gamma \cdot (u,g) \rightarrow (\gamma u, g), \quad u \in Y\backslash R, \gamma \in \Gamma_y.
\end{align*}
Therefore, a $(\Gamma,G)$-torsor $F$ being of type $\boldsymbol\rho$, is equivalent to giving $(\Gamma,G)$-isomorphisms
\begin{align*}
\Theta_y: F_y|_{\mathbb{D}^{\times}_y} \rightarrow F_0|_{\mathbb{D}^{\times}_y}, \quad y \in R.
\end{align*}
With respect to the local picture we discussed, a $(\Gamma,G)$-bundle $F$ over $Y$ corresponds to a $\mathcal{G}_{\boldsymbol\theta}$-torsor over $X$.
\begin{thm}[Theorem 5.3.1 in \cite{Moduli of parahoric torsors}]\label{equi for torsors}
The category of $(\Gamma,G)$-torsors of type $\boldsymbol\rho$ over $Y$ is equivalent to the category of $\mathcal{P}_{\boldsymbol\theta}$-torsors over $X$. Furthermore, the equivalence of categories gives the equivalence of stacks, i.e.
\begin{align*}
{\rm Bun}^{\boldsymbol\rho}_G([Y/\Gamma]) \cong {\rm Bun}_{\mathcal{P}_{\boldsymbol\theta}}(X).
\end{align*}
\end{thm}

\subsection{Logarithmic Higgs Bundles}\label{subsect Higgs}
Let $F$ be a $(\Gamma,G)$-bundle over $\mathbb{D}_y$, and denote by $E$ the corresponding $\mathcal{P}_\theta(\mathcal{O}_z)$-torsor over $\mathbb{D}_x$. Denote by ${\rm Ad}(F)$ the adjoint bundle. Without loss of generality, suppose that ${\rm Ad}(F)=\mathfrak{g}(\mathcal{O}_w)$. Let $\phi$ an element in $\mathfrak{g}(\mathcal{O}_w) \frac{dw}{w}$, which is regarded as a section $\mathbb{D}_y \rightarrow {\rm Ad}(F) \otimes \Omega_{\mathbb{D}_y}(y)$ and is called a \emph{logarithmic Higgs field}. Assume that $\phi$ is $\Gamma$-equivariant, i.e.
\begin{align*}
\phi(\gamma w)=\rho(\gamma) \phi(w) \rho^{-1}(\gamma).
\end{align*}
With the same notations as in \S\ref{subsect torsor}, we define
\begin{align*}
\varphi=\Delta^{-1} \phi \Delta.
\end{align*}
Clearly, $\varphi$ is $\Gamma$-invariant, i.e.
\begin{align*}
\varphi(\gamma w)=\varphi(w).
\end{align*}
Therefore, $\varphi(w)$ can be descended to a section $\mathbb{D}_x \rightarrow {\rm Ad}(E) \otimes \Omega_{\mathbb{D}_x}(x)$ by substituting $z=w^d$, which implies that $\varphi(z) \in \mathfrak{p}_\theta(\mathcal{O}_z)\frac{dz}{z}$. Abusing the notation, we still use $\varphi: \mathbb{D}_x \rightarrow {\rm Ad}(E)\otimes \Omega_{\mathbb{D}_x}(x)$ for the corresponding section. It is easy to check that a $\Gamma$-equivariant logarithmic Higgs field $\phi$ of $F$ corresponds to a unique logarithmic $\mathcal{P}_\theta(\mathcal{O}_z)$-Higgs field $\varphi$ of $E$. With respect to the above discussion, we have a one-to-one correspondence between logarithmic $(\Gamma,G)$-Higgs bundles over $\mathbb{D}_y$ and logahoric $\mathcal{P}_{\theta}$-Higgs bundles over $\mathbb{D}_x$. This local discussion can be generalized globally. With the same setup as in \S\ref{subsect torsor}, we introduce the following definition.
\begin{defn}
A \emph{logarithmic $(\Gamma,G)$-Higgs bundle} over $Y$ is a pair $(F,\phi)$, where $F$ is a $(\Gamma,G)$-torsor over $Y$ and $\phi \in H^0(Y,{\rm Ad}(F)\otimes \Omega_Y(R))$ is a $\Gamma$-equivariant section.
\end{defn}

\begin{thm}[Theorem 3.6 in \cite{Parahoric Higgs}]\label{equivalence for higgs bundles}
The stack of logarithmic $(\Gamma,G)$-Higgs bundles of type $\boldsymbol\rho$ over $Y$ is isomorphic to the stack of logahoric $\mathcal{P}_{\boldsymbol\theta}$-Higgs bundles over $X$, i.e.
\begin{align*}
{\rm Higgs}^{tame,{\boldsymbol\rho}}_{G}([Y/\Gamma]) \cong {\rm Higgs}^{tame}_{\mathcal{P}_{\boldsymbol\theta}}(X).
\end{align*}
\end{thm}
We refer the reader to \cite{Sun1901,Sun2003,Global Springer} for more details about the algebraic stacks ${\rm Higgs}^{tame}_{\mathcal{P}_{\boldsymbol\theta}}(X)$ and ${\rm Higgs}^{tame,{\boldsymbol\rho}}_{G}([Y/\Gamma])$.

\subsection{Tame Local Systems}\label{subsect_tame_local_sys}
We follow the same notations as above.
\begin{defn}
A \emph{tame $(\Gamma,G)$-local system of type $\rho$} over $Y$ is a $(\Gamma, G)$-torsor $F$ together with a logarithmic $\Gamma$-equivariant $G$-connection $\nabla$. A logarithmic $\Gamma$-equivariant $G$-connection is also called a \emph{logarithmic $(\Gamma,G)$-connection}.
\end{defn}

Let $y$ be a ramification point on $Y$, $x$ be its image on $X$. Over a neighborhood of $y$, the connection $\nabla$ can be written as $d - A\frac{dw}{w}$, where $A \in \mathfrak{g}(\mathcal{O}_w)$. Then the condition that $\nabla$ is $\Gamma$-invariant (under the gauge action) means that
\begin{align*}
d\rho(\gamma) \rho(\gamma)^{-1}+{\rm Ad}(\rho(\gamma)) A(w) \frac{dw}{w}=A(\gamma w)\frac{dw}{w},
\end{align*}
where $\rho : \Gamma \rightarrow T$ is the representation we studied in \S\ref{subsect torsor}. Since $\rho(\gamma)$ is a constant matrix, therefore we have
\begin{align*}
{\rm Ad}(\rho(\gamma)) A(w) \frac{dw}{w}=A(\gamma w)\frac{dw}{w},
\end{align*}
which is equivalent to say that $A$ lies in $\mathfrak{g}(\mathcal{O}_{w})^{\Gamma}=\mathfrak{p}_{\theta}(\mathcal{O}_{z})$. With the same idea as Theorem \ref{equi for torsors} and \ref{equivalence for higgs bundles}, we get the following:
\begin{prop}\label{parahoric connection equivariant}
The stack of tame $(\Gamma,G)$-local systems of type $\boldsymbol\rho$ over $Y$ is equivalent to the stack of tame $\mathcal{P}_{\boldsymbol\theta}$-local systems over $X$, i.e.
\begin{align*}
{\rm Locsys}^{tame, {\boldsymbol\rho}}_G([Y/\Gamma]) \cong {\rm Locsys}^{tame}_{\mathcal{P}_{\boldsymbol\theta}}(X).
\end{align*}
\end{prop}

\begin{proof}
We only give the correspondence locally. As we discussed in \S\ref{subsect torsor}, a $(\Gamma,G)$-torsor $F$ of type $\rho$ over $\mathbb{D}_y$ corresponds to a unique $\mathcal{P}_\theta(\mathcal{O}_z)$-torsor over $\mathbb{D}_x$. A tame $(\Gamma,G)$-local system on $\mathbb{D}_y$ can be regarded as a logarithmic connection $A\frac{dw}{w}$, where $A \in \mathfrak{g}(\mathcal{O}_{w})$. Under the gauge action by $\Delta^{-1}$, the logarithmic $(\Gamma,G)$-connection $A\frac{dw}{w}$ corresponds to following
\begin{align*}
\Delta^{-1} \ast A\frac{dw}{w}=d \cdot (- \theta +{\rm Ad}(\Delta^{-1})A) \frac{dw}{w}.
\end{align*}
Note that ${\rm Ad}(\Delta^{-1})A$ is $\Gamma$-invariant, which can be descended to an element in $\mathfrak{p}_{\theta}(\mathcal{O}_{z})$ by substituting $z=w^d$. As we explained above, the element $- \theta +{\rm Ad}(\Delta^{-1})A$ is in $\mathfrak{p}_{\theta}(\mathcal{O}_{z})$. Therefore, we get a logarithmic $\mathcal{P}_{\theta}(\mathcal{O}_{z})$-connection on $\mathbb{D}_{x}$. The other direction can be proved similarly.
\end{proof}


\section{Local Tame Parahoric Nonabelian Hodge Correspondence}\label{sect local}
In this section, we study tame $G$-local systems and logarithmic $G$-Higgs bundles over a formal disc $\mathbb{D}$. We prove the \emph{tame nonabelian Hodge correspondence} over $\mathbb{D}$ and generalize this correspondence to a parahoric version.

Fixing a local coordinate $z$ of $\mathbb{D}={\rm Spec}(\mathcal{O})$, a tame $G$-local systems over $\mathbb{D}$ is a logarithmic connection $d - A\frac{dz}{z}$ (indeed over the trivial $G$-torsor) , where $A \in \mathfrak{g}(\mathcal{O})$. We say that $d - A\frac{dz}{z}$ is equivalent to $d - B\frac{dz}{z}$ (under the gauge action) if there exists $g \in G(\mathcal{O})$ such that
\begin{align*}
dg g^{-1}+{\rm Ad}(g)A\frac{dz}{z}=B\frac{dz}{z}.
\end{align*}
This action is known as the \emph{gauge action}, and denote it by $g \ast A\frac{dz}{z}=B\frac{dz}{z}$. Therefore, the category of logarithmic $G$-connections over $\mathbb{D}$ is regarded as the equivalence classes of logarithmic $G$-connections $d-A\frac{dz}{z}$ under the gauge action.

Now we come to the side of logarithmic $G$-Higgs bundles. A logarithmic $G$-Higgs field over $\mathbb{D}$ is $A\frac{dz}{z}$, where $A \in \mathfrak{g}(\mathcal{O})$. The action of $G(\mathcal{O})$ on Higgs fields is defined as the \emph{adjoint action}, and $A\frac{dz}{z}$ is equivalent to $B\frac{dz}{z}$ if there exists $g \in G(\mathcal{O})$ such that $A={\rm Ad}(g)B$. Thus, the category of logarithmic $G$-Higgs bundles over $\mathbb{D}$ is considered as the equivalence classes of Higgs fields $A\frac{dz}{z}$ under the adjoint action.

\subsection{Standard Form of Tame Connections}
In this subsection we would like to prove an analogue of the results in \cite{Riemann Hilbert} in positive characteristic. For the case of characteristic zero, we also refer the reader to \cite{Herro} for more details. First let us consider logarithmic connections for the reductive group $G$.
\begin{lem}\label{exp}
Let $\theta$ be a weight. Let $X_k \in L\mathfrak{g}_{\geq k}$ where $k>0$. Then there exists an element $g_k \in \mathcal{G}_{\geq k}$ such that
\begin{align*}
{\rm Ad}(g_{k})={\rm id}+[X_{k},-] \mod L\mathfrak{g}_{>k}
\end{align*}
as an operator on $L\mathfrak{g}_{\geq 0}$.
\end{lem}

\begin{proof}
Let $X_{k}=\sum_{\alpha} X^{\alpha}_{k}$, where $X^{\alpha}_{k}\in\mathfrak{g}_{\alpha}(\mathcal{K})$. Since $k>0$, it is easy to verify that if we have constructed $g^{\alpha}_{k}$ for each $X^{\alpha}_{k}$, then the product $g_k=\prod g^{\alpha}_{k}$ satisfies the requirement. So we can assume $X_{k}=X^{\alpha}_{k}\not\in L\mathfrak{g}_{>k}$. Furthermore, it is enough to consider the case when $X^{\alpha}_{k}$ is of the form $z^{k-\langle \theta,\alpha \rangle}\mathfrak{g}_{\alpha}$. If $\alpha\neq 0$, then one can consider a morphism from $SL_{2}$ to $G(\mathcal{K})$ corresponds to the root $\alpha+k-\langle \theta,\alpha \rangle$, then the statement essentially reduces to $SL_{2}$ representations, so we can choose $g^{\alpha}_{k}\in U_{\alpha}(z^{-\langle \theta,\alpha \rangle+k}\mathcal{O})$ that satisfies our requirement if $\alpha\neq 0$ . If $\alpha=0$, we can choose $g^{\alpha}_{k}\in T(\mathcal{O})$ that satisfies our requirement.
\end{proof}

\begin{lem}[Standard Form]\label{standard form 1}
Let $A=\sum_{i\geq 0} a_{i}z^i\in\mathfrak{g}(\mathcal{O})$. There exists $g\in G(\mathcal{O})$ such that
\begin{align*}
dgg^{-1}+{\rm Ad}(g)A\frac{dz}{z}= B\frac{dz}{z}
\end{align*}
where $B=\sum_{i\geq 0}b_{i}z^i$ and $b_{i}$ lies in the generalized eigenspace of the operator $[b_{0},-]$ with eigenvalue $i$.
\end{lem}

\begin{proof}
We shall construct a sequence of elements $g_{k}$, where $g_{k}$ lies in the $k$-th congruence subgroup of $G(\mathcal{O})$ such that if we take $g=g_{k}g_{k-1}\cdots g_{0}$, then after applying gauge action of $g$, we have
\begin{align*}
g \ast A\frac{dz}{z}=(\sum_{i\geq 0}b_{i}z^i)\frac{dz}{z},
\end{align*}
such that $b_{i}$ lies in the generalized eigenspace of $b_{0}$ with eigenvalue $i$ for all $i\leq k$. For $k=0$, we can take $g_{0}=1$. Suppose now we have chosen the elements $g_{0},g_{1},\cdots,g_{k} \in G(\mathcal{O})$ such that we get $(\sum_{i\geq 0}b_{i} z^i)\frac{dz}{z}$, where $b_{i}$ lies in the generalized eigenspace of $b_{0}$ with eigenvalue $i$ for all $i\leq k$. Abusing the notation, let
\begin{align*}
A=b_0+b_1 z + \dots+ b_k z^k + a_{k+1} z^{k+1} + \dots \text{ .}
\end{align*}
Taking an arbitrary element $y \in\mathfrak{g}$, let $x\in G(\mathcal{O})$ be the element satisfying the condition of Lemma \ref{exp} for $y z^{k+1}$. Note that
\begin{align*}
dx x^{-1}=(k+1)yz^{k+1}\frac{dz}{z} \mod z^{k+2}\mathfrak{g}(\mathcal{O})\frac{dz}{z}, \quad {\rm Ad}(x)A=(A+[yz^{k+1},A])\frac{dz}{z} \mod z^{k+2} \mathfrak{g}(\mathcal{O})\frac{dz}{z}.
\end{align*}
After gauge transform by $x$, we get
\begin{align*}
x \ast A= dx x^{-1}+{\rm Ad}(x)A\frac{dz}{z}&=\left( - (k+1)yz^{k+1}+ (A+[yz^{k+1},A])\mod z^{k+2} \mathfrak{g}(\mathcal{O})\right)\frac{dz}{z}\\
&=(\sum_{i=0}^k b_i z_i) \frac{dz}{z}+ \left( - (k+1)y+a_{k+1}+[y,b_0]   \right) z^{k+1} \frac{dz}{z} \mod z^{k+2} \mathfrak{g}(\mathcal{O})\frac{dz}{z}.
\end{align*}
We can choose an element $y \in\mathfrak{g}$ such that $(k+1)y+a_{k+1}+[y,b_0]$ lies in the generalized eigenspace of $b_0$ with eigenvalue $(k+1)$. This finishes the proof of this lemma.
\end{proof}

Given any logarithmic connection $d-A \frac{dz}{z}$, where $A=\sum_{i \geq 0} a_i z^i$, we can assume that under gauge transformation, $a_i$ lies in the generalized eigenspace of the operator $[a_0,-]$ by Lemma \ref{standard form 1}, and a logarithmic connection in this form will be called \emph{in standard form}. We will also need a more general version of Lemma~\ref{standard form 1} which works over an Artinian local algebra over $k$:
\begin{lem}\label{standard form over artinian rings}
Let $(R,\mathfrak{m})$ be an Artinian local algebra over $k$. Let $A=\sum_{i \geq 0} a_i z^i \in \mathfrak{g}(\mathcal{O}_R)$ be an element, where $\mathcal{O}_R:=R[[z]]$. Suppose that $a_{0}\in\zeta+\mathfrak{m}\otimes\mathfrak{g}$, where $\zeta\in\mathfrak{g}$. Let $\mathfrak{g}_{\lambda}$ be the generalized eigenspace of adjoint action of $\zeta$ with eigenvalue $\lambda$. Then there exists $g\in G(\mathcal{O}_R)$ such that
\begin{align*}
dgg^{-1}+{\rm Ad}(g)A\frac{dz}{z}=B\frac{dz}{z},
\end{align*}
where $B=\sum_{i \geq 0} b_{i}z^{i}$ and each $b_{i}$ lies in the $\mathfrak{g}_{i}\otimes R$.
\end{lem}

\begin{proof}
There exists a filtration of $R$ by ideals $I_{i}$ such that $I_{i}/I_{i+1}$ is annihilated by $\mathfrak{m}$. Using this filtration, one can show that for any element $y\in\mathfrak{g}_{\lambda}\otimes R$ and any integer $n\neq\lambda$, there exists $P\in\mathfrak{g}\otimes R$ such that $y=nP-[\zeta,P]$. Now one can use the same argument as in Lemma~\ref{standard form 1} to finish the proof.
\end{proof}

\subsection{Irrational Case}
Let $d - A\frac{dz}{z}$ be a logarithmic connection, where $A=\sum_{i \geq 0} a_i z^i$. Let $a_0=\tau+\sigma+n$, where $\tau$ is the rational part and $\sigma$ is the irrational part (see \S\ref{sect pre}). We first consider a special case that the eigenvalues of $[a_0,-]$ are not rational, i.e. $\tau=0$.

\begin{cor}\label{standard form of irr connection}
Let $A=\sum_{i \geq 0} a_i z^i$, where $a_0=\sigma+n$ with trivial rational part ($\tau=0$). Then there exists $g\in G(\mathcal{O})$ such that $dgg^{-1}+{\rm Ad}(g)A\frac{dz}{z}\in\mathfrak{g}(\mathcal{O}^p)\frac{dz}{z}$.
\end{cor}

\begin{proof}
This is a direct result of Lemma \ref{standard form 1}.
\end{proof}

This corollary shows that a logarithmic connection in this case is equivalent to a logarithmic connection in $\mathfrak{g}(\mathcal{O}^p) \frac{dz}{z}$, which is in standard form.  On the other hand, if $a_{0}$ has non-zero rational eigenvalues, then it is unclear whether $A$ can be put into $\mathfrak{g}(\mathcal{O}^p)$ via $G(\mathcal{O})$. But nonetheless one can show that $A$ can be put into $\mathfrak{g}(\mathcal{K}^p)\frac{dz}{z}$ via $G(\overline{\mathcal{K}})$:

\begin{lem}\label{standard form 2}
Let $A=\sum_{i\geq 0} a_{i}z^i\in\mathfrak{g}(\mathcal{O})$. Then there exists $g\in G(\overline{\mathcal{K}})$ such that
\begin{align*}
dgg^{-1}+{\rm Ad}(g)A\frac{dz}{z}=C\frac{dz}{z},
\end{align*}
where $C\in\mathfrak{g}(\mathcal{K}^p)$. Moreover, if we choose a splitting of the embedding $\mathfrak{t}_{\mathbb{F}_{p}}\hookrightarrow \mathfrak{t}$ as well as a set of representatives for the semisimple orbits for the adjoint action of $G$ on $\mathfrak{g}$ denoted by $\mathcal{D}$, then one can assume that $(C_{0})_{ss}\in\mathcal{D}$.
\end{lem}

\begin{proof}
By Lemma~\ref{standard form 1}, we can assume that each $a_{i}$ lies in the generalized eigenspace of $a_{0}$ with eigenvalue $i$. Let $a_0=\tau+\sigma+n$, where $\tau$ is the rational part and $\sigma$ is the irrational part. Then we have $[\tau,a_{i}]= i a_{i}$. One may choose a weight $\theta_\tau \in X_*(T) \otimes_{\mathbb{Z}} \mathbb{Q}$ such that
\begin{align*}
\langle \theta_\tau,\alpha \rangle \in\mathbb{Z} \quad \text{ and } \quad \langle \theta_\tau,\alpha \rangle=\alpha(\tau) \text{ mod } p,
\end{align*}
for any root $\alpha$. Now taking the gauge action by $z^{\theta_\tau}$, we get the desired form.
\end{proof}

\begin{lem}\label{automorphisms}
If $g \in G(\mathcal{K})$ such that $g \ast C\frac{dz}{z}=D\frac{dz}{z}$, where $C$ and $D$ are of the form in Lemma~\ref{standard form 2}. Then $g \in G(\mathcal{K}^p)$.
\end{lem}

\begin{proof}
We may assume that the semisimple part of $c_{0}$ and $d_{0}$ are in $\mathfrak{t}$, and we also choose an embedding of $G$ into $GL_{n}$. Thus, we shall view $g$, $C$, $D$ as elements in $GL_{n}(\mathcal{O})$ and $\mathfrak{g}\mathfrak{l}_{n}(\mathcal{O})$. Write $g=\sum_{i\geq k}g_{i}z^i$. Let $j$ be the smallest index such that $j$ does not divide $p$ and $g_{j}\neq 0$. Then using equation
\begin{align*}
dg+gA=Bg,
\end{align*}
we conclude that $jg_{j}+g_{j}c_{0}=d_{0}g_{j}$. Consider the operator on $\mathfrak{g}\mathfrak{l}_{n}$ given by
\begin{align*}
T \rightarrow d_{0}T-T c_{0}.
\end{align*}
Since the semisimple parts of $a_{0}$ and $b_{0}$ are in $\mathfrak{t}$ and that $a_{0}$ and $b_{0}$ has no nonzero rational eigenvalues, we conclude that the only rational eigenvalue of this operator is zero. Therefore, $j=0$ and then $j$ is divisible by $p$.
\end{proof}

Let ${\rm Locsys}^{tame}_{G,irr}(\mathbb{D})$ be the category of tame $G$-local systems on the formal disc $\mathbb{D}$ such that the semisimple part of the residue is irrational. Similarly, let ${\rm Higgs}^{tame}_{G,irr}(\mathbb{D}')$ be the category of logarithmic $G$-Higgs bundles on $\mathbb{D}'$, the Frobenius twist of $\mathbb{D}$, such that the semisimple part of the residue is irrational.

\begin{prop}\label{standard form in irrational}
The category ${\rm Locsys}^{tame}_{G,irr}(\mathbb{D})$ is equivalent to ${\rm Higgs}^{tame}_{G,irr}(\mathbb{D}')$.
\end{prop}

\begin{proof}
Take an element $A\frac{dz}{z} \in {\rm Locsys}^{tame}_{G,irr}(\mathbb{D})$, where $A=\sum_{i \geq 0} a_i z^i$. Since the semisimple part of $a_0$ is irrational, this element $A\frac{dz}{z}$ is equivalent to $A'\frac{dz}{z} \in \mathfrak{g}(\mathcal{O}^p) \frac{dz}{z}$ under the gauge action by Corollary \ref{standard form of irr connection}. By the Frobenius twist, the logarithmic connection $A'\frac{dz}{z}$ can be considered as a logarithmic Higgs field over $\mathbb{D}'$, and we use the same notation $A'\frac{dw}{w}$. Furthermore, by Lemma \ref{automorphisms}, we get a well-defined functor
\begin{align*}
{\rm Locsys}^{tame}_{G,irr}(\mathbb{D}) & \rightarrow {\rm Higgs}^{tame}_{G,irr}(\mathbb{D}') \\
A & \rightarrow A' .
\end{align*}
It is easy to check that this functor induces an equivalence of categories.
\end{proof}

\subsection{Rational Case}
For any semisimple rational element $\tau \in \mathfrak{g}$, we equip the trivial $G$-torsor over $\mathbb{D}$ with a logarithmic $G$-connection given by $\tau \frac{dz}{z}$. Let $G'_{\tau}(\mathcal{O}) \subset G(\mathcal{O})$ be the group of automorphisms of $d - \tau\frac{dz}{z}$. Moreover, the elements in the Lie algebra $\mathfrak{g}'_{\tau}(\mathcal{O})$ of $G'_{\tau}(\mathcal{O})$ can be written as $A=\sum_{i\geq 0} a_{i}z^i$, where each $a_{i}$ lies in the generalized eigenspace of $\tau$ with eigenvalue $i$. Since for any $g \in G'_{\tau}(\mathcal{O})$, we have
\begin{align*}
dgg^{-1}+g\tau g^{-1}\frac{dz}{z}=\tau\frac{dz}{z},
\end{align*}
then gauge action of $G'_{\tau}(\mathcal{O})$ on its Lie algebra can be written as
\begin{align*}
g*A \frac{dz}{z}= dgg^{-1}+gAg^{-1} \frac{dz}{z}= \tau \frac{dz}{z}+g(A-\tau)g^{-1} \frac{dz}{z}.
\end{align*}

\begin{lem}\label{parahoric of rational semisimple}
The automorphism group $G'_{\tau}(\mathcal{O})$ can be identified with a parahoric subgroup over $\mathbb{D}'$.
\end{lem}

\begin{proof}
With the same approach as Lemma \ref{standard form 2}, let $\theta_\tau$ be a weight such that
\begin{align*}
\langle \theta_\tau,\alpha \rangle \in\mathbb{Z} \quad \text{ and } \quad \langle \theta_\tau,\alpha \rangle=\alpha(\tau) \text{ mod } p.
\end{align*}
Then from the definition of $G'_{\tau}(\mathcal{O})$, it is easy to see that the Lie algebra of ${\rm Ad}(z^{-\theta_\tau})(G'_{\tau}(\mathcal{O}))$ is the Lie algebra of the parahoric subgroup of $G(k((z^p)))$ over $\mathbb{D}'$ defined by the weight $\frac{\theta_\tau}{p}$. It remains to show $G'_{\tau}(\mathcal{O})$ is connected. It is easy to see that the image of the evaluation morphism $G'_{\tau}(\mathcal{O})\rightarrow G(\mathcal{O})\rightarrow G$ is equal to $Z_{G}(\tau)$, which is a Levi subgroup of $G$, hence connected. This proves the claim.

\end{proof}

\begin{lem}\label{normal form with rational}
Take two elements
\begin{align*}
A=\sum_{i\geq 0} a_{i}z^i \in \mathfrak{g}(\mathcal{O}), \quad B=\sum_{i\geq 0} b_{i}z^i \in \mathfrak{g}(\mathcal{O})
\end{align*}
in standard form. Assume $a_{0}$ and $b_{0}$ has the same rational semisimple part $\tau$. Suppose that $g \in G(\mathcal{O})$ such that $dgg^{-1}+gAg^{-1}\frac{dz}{z}=B\frac{dz}{z}$, then $g\in G'_{\tau}(\mathcal{O})$.
\end{lem}

\begin{proof}
First, we represent $g$ as $g=g_{1}g_{0}$ where $g_0 \in G$ and $g_{1}$ lies in the kernel of the evaluation map $G(\disc)\rightarrow G$. By the condition $dgg^{-1}+gAg^{-1}\frac{dz}{z}=B\frac{dz}{z}$, one concludes that $g_{0} \in Z_{G}(\tau)$. Since $Z_{G}(\tau)\subseteq G'_{\tau}(\disc)$, one may assume that $a_{0}=b_{0}$.

Now we fix a representation of $G$ and view $A$ as a matrix. The condition $dgg^{-1}+gAg^{-1}\frac{dz}{z}=B\frac{dz}{z}$ gives us
\begin{align*}
dg+gA\frac{dz}{z}=Bg\frac{dz}{z}.
\end{align*}
By the property of $G'_\tau(\mathcal{O})$, the element $g \in G'_{\tau}(\mathcal{O})$ if and only if $[\tau,g_{i}]=i g_{i}$ when $g_i \neq 0$, where $g=\sum_{i \geq 0}g_iz^i$. Let $k$ be the smallest index such that $g_{k}\neq 0$ and $g_{k}$ is not in the eigenspace of $\tau$ with eigenvalue $k$. Then we have
\begin{align*}
kg_{k}+\sum_{i+j=k}g_{i}a_{j}=\sum_{i+j=k}b_{j}g_{i}.
\end{align*}
If $i<k$ and $a_{j}g_{i}\neq 0$ or $g_{i}b_{j}\neq 0$, then by our choice of $k$, $a_{j}p_{i}$ satisfies
\begin{align*}
[\tau ,a_{j}g_{i}]=(i+j)a_{j}g_{i}=ka_{j}g_{i}.
\end{align*}
A similar formula also holds for $g_{i}b_{j}$. Thus, the element $kg_{k}+g_{k}a_{0}-a_{0}g_{k}$ lies in the generalized eigenspace of $a_{0}$ with eigenvalue $k$. By the assumption that $g_{k}$ is not in this subspace, we arrives at a contradiction.
\end{proof}

We fix a rational semisimple element $\tau$ in $\mathfrak{g}$. Let ${\rm Locsys}^{tame}_{G,\tau}(\mathbb{D})$ be the category of tame $G$-local systems on the formal disc such that the rational semisimple part of the residue is conjugate to $\tau$. More precisely, if $d-A\frac{dz}{z} \in {\rm Locsys}^{tame}_{G,\tau}(\mathbb{D})$ where $A=\sum_{i \geq 0} a_i z^i$, then the rational semisimple part of $a_0$ is conjugate to $\tau$. Let ${\rm Higgs}^{tame}_{G'_\tau, irr}(\mathbb{D}')$ be the category of logarithmic $G'_\tau(\mathcal{O})$-Higgs bundles on the Frobenius twist of $\mathbb{D}$ such that the semisimple part of the residue is irrational (up to conjugation). Here we regard $G'_\tau(\mathcal{O})$ as the parahoric group $\mathcal{P}_{\frac{\theta_\tau}{p}}(\mathcal{O}')$ determined in Lemma \ref{parahoric of rational semisimple}.

\begin{prop}\label{standard form in general}
The category ${\rm Locsys}^{tame}_{G,\tau}(\mathbb{D})$ is equivalent to ${\rm Higgs}^{tame}_{G'_\tau, irr}(\mathbb{D}')$. Moreover, let $\Gamma$ be a cyclic group with order $d$ coprime to $p$, and then the category ${\rm Locsys}^{tame}_{G,\tau}([\mathbb{D} / \Gamma])$ is equivalent to ${\rm Higgs}^{tame}_{G'_\tau, irr}([\mathbb{D}'/ \Gamma])$.
\end{prop}

\begin{proof}
As we discussed at the beginning of this subsection, we find that
\begin{align*}
g*A \frac{dz}{z}= dgg^{-1}+gAg^{-1} \frac{dz}{z}= \tau \frac{dz}{z}+g(A-\tau)g^{-1} \frac{dz}{z},
\end{align*}
where $g \in G'_{\tau}(\mathcal{O})$. We use the same approach as in Proposition~\ref{standard form in irrational} to construct a map
\begin{align*}
{\rm Locsys}^{tame}_{G,\tau}(\mathbb{D}) \rightarrow {\rm Higgs}^{tame}_{G'_\tau, irr}(\mathbb{D}'), \quad A \rightarrow {\rm Ad}(z^{-\theta_\tau}) (A-\tau).
\end{align*}
By Lemma~\ref{normal form with rational} as well as the action of $G'_{\tau}(\mathcal{O})$ on its Lie algebra, it is easy to check that this map is well-defined and bijective.

For the version of stacks, it is equivalent to consider logarithmic $(\Gamma,G)$-connection on $\mathbb{D}$ such that the rational semisimple part of the residue is $\tau$. As we discussed in \S\ref{subsect_tame_local_sys}, a logarithmic $G$-connection $d - A \frac{dz}{z}$ is $\Gamma$-equivariant means that $A \in \mathfrak{g}(\mathcal{O})$ is $\Gamma$-equivariant. Clearly, $(A-\tau) \frac{dz}{z}$ is also $\Gamma$-equivariant as a logarithmic Higgs field. Therefore, the correspondence also holds for stacks.
\end{proof}

For future use, let us note the following:
\begin{lem}\label{normal form for connection with zero p curvature}
Let $A=\sum_{i \geq 0} a_i z^i$ be an element in $\mathfrak{g}'_\tau(\mathcal{O})\frac{dz}{z}$, where $\tau$ is the rational part of the $a_0$. Then the logarithmic $G$-connection $d - A\frac{dz}{z}$ has zero $p$-curvature if and only if $A=\tau$. More generally, let $A\in\mathfrak{g}(\mathcal{O}_R)$, where $(R,\mathfrak{m})$ is an Artinian local algebra over $k$. Suppose that the rational part of $a_{0}$ is $\tau$ when modulo $\mathfrak{m}$. Then the connection $d - A\frac{dz}{z}$ has zero $p$-curvature if and only if there exists $g\in G(\mathcal{O}_R)$ such that $dgg^{-1}+gAg^{-1}\frac{dz}{z}=\tau \frac{dz}{z}$.
\end{lem}

\begin{proof}
It is well-known that in the logarithmic case, we have $(z\partial)^{(p)} = z\partial$ \cite[\S 1.2.2]{Ogu94}. Suppose that the logarithmic connection is $d - A \frac{dz}{z}$. It has zero $p$-curvature if and only if $A^p - A = 0$. Based on this fact, it is clear that when $A=\tau$, the logarithmic connection $d - A\frac{dz}{z}$ has zero $p$-curvature. Then we consider the other direction. Suppose that $d - A \frac{dz}{z}$ has zero $p$-curvature, where $A \in \mathfrak{g}'_\tau(\mathcal{O})$ with rational part $\tau$. With the same approach as Lemma \ref{standard form 2} and \ref{parahoric of rational semisimple}, we choose a weight $\theta_\tau \in X_*(T) \otimes_{\mathbb{Z}} \mathbb{Q}$. Let $B=z^{-\theta_\tau}\ast A\frac{dz}{z}={\rm Ad}(z^{-\theta_{\tau}})(A-\tau)\frac{dz}{z}$. Clearly, the logarithmic connection $d - B\frac{dz}{z}$ also has zero $p$-curvature. Note that $B$ is of the form $\sum b_{i}z^{pi}$ by Proposition~\ref{standard form in irrational}. Thus, $(z\partial)(B)=0$. Together with the fact $B^p-B=0$, one conclude that $B=0$. Therefore, we have $A=\tau$.

In the general situation, one can first use Lemma~\ref{standard form over artinian rings} to convert $A$ into the form
\begin{align*}
A= \sigma+\textrm{higher degree terms } \in \mathfrak{g}'_{\tau}(\mathcal{O}_R),
\end{align*}
where the semisimple part of $\sigma$ modulo $\mathfrak{m}$ is irrational. Then the condition that $p$-curvature of $A$ equals to zero implies that $A=\tau$ by the same argument as in the case over a field.
\end{proof}

\subsection{Parahoric Case}

In this subsection, we would like to generalize the results above to the case of parahoric subgroups. Namely, let $\parahwt(\mathcal{O})$ be a parahoric subgroup of $G(\mathcal{K})$ with a tame weight $\theta \in X_{*}(T) \otimes_{\mathbb{Z}}\mathbb{Q}$ and the weight $\theta$ can be regarded as an element in $\mathfrak{t}_{\mathbb{F}_{p}}$. In a similar way, we consider the gauge action of $\parahwt(\mathcal{O})$ on $\mathfrak{p}_\theta(\mathcal{O}) \frac{dz}{z}$. Then one may adapt the analysis on $\mathfrak{g}(\mathcal{O})\frac{dz}{z}$ to $\mathfrak{p}_\theta(\mathcal{O}) \frac{dz}{z}$, but we shall use a slightly different approach which is based on \cite{Moduli of parahoric torsors}.

Recall that a weight $\theta$ induces a natural decomposition of $\mathfrak{g}=\bigoplus\limits_{\lambda \in \mathbb{Q}} \mathfrak{g}_{\lambda}$ indexed by $\mathbb{Q}$, where $\mathfrak{g}_{\lambda}$ is the $\lambda$-th graded piece. This decomposition also induces a filtration of the loop Lie algebra $L\mathfrak{g}$ (see \S\ref{sect pre}). The following is a reformulation of the proof of \cite[Theorem 6]{Riemann Hilbert} in our setting:
\begin{lem}\label{standard form 1 parahoric}
Let $A\in\lieparahwt(\disc)\frac{dz}{z}$ be an element. Suppose that the weight zero component $A(0)$ is equal to $(\tau+\sigma+\sum a_{i}z^i)\frac{dz}{z}$, where $\tau$ is semisimple rational element and $\sigma$ is a semisimple irrational element such that
\begin{align*}
[\tau, a_{i}]=ia_{i}, \quad [\sigma, a_{i}]=0, \quad \sum a_{i} \text{ is nilpotent.}
\end{align*}
Then there exists $g\in\parahwt(\disc)$ such that
\begin{align*}
dgg^{-1}+{\rm Ad}(g)A=(\tau+\sigma+\sum A_{i}z^{i})\frac{dz}{z}
\end{align*}
where $A_{i}\in\mathfrak{g}_{\geq i}$, $[\tau, A_{i}]=iA_{i}$, $[\sigma, A_{i}]=0$.
\end{lem}

Now we take an element in $A=\sum_{i \geq 0} a_i z^i \in \mathfrak{p}_\theta(\mathcal{O}_z)$, where we use $z$ to emphasize the local coordinate. By Proposition \ref{parahoric connection equivariant}, the gauge action of $\parahwt(\mathcal{O}_z)$ on the Levi quotient $\mathfrak{l}_\theta(\mathcal{O}_z)$ gets transformed into adjoint action of $G(\mathcal{O}_w)$ on $\mathfrak{g}$. Fixing a choice of $d$-th root of unity $\zeta$, we get the following:

\begin{lem}\label{correspondence between orbits}
There is a one-to-one correspondence between orbits of $\mathfrak{l}_\theta(\mathcal{O}_z)\frac{dz}{z}$ under the gauge action of $\mathcal{P}_\theta(\mathcal{O}_z)$ and the orbits of $\mathfrak{h}_\theta\frac{dw}{w}$ under the adjoint action of $Z_{G}(\zeta^{d\theta})$.
\end{lem}

\begin{proof}
The proof is similar to \cite[Lemma 4]{Riemann Hilbert}.
\end{proof}

Given an element $A \in \mathfrak{p}_\theta(\mathcal{O}_z)$, let $A=\sum_{i \geq 0} a_i z^i$, where $a_0=\tau+\sigma$, and denote by $d+ A \frac{dz}{z}$ the corresponding connection on $\mathbb{D}_x$. Then, by Lemma \ref{standard form 1 parahoric} and \ref{correspondence between orbits}, it corresponds to a $\Gamma$-equivariant connection $d(\theta+\tau+\sigma+\sum a_i) \frac{dw}{w}$ on $\mathbb{D}_y$. Now we choose a tame weight $\theta_\tau$ with the property that
\begin{align*}
\langle \theta_\tau, \alpha \rangle \equiv \alpha(\tau) \text{ mod } p
\end{align*}
for all roots $\alpha$ and that $\langle \theta_\tau, \alpha \rangle =0$ whenever $\alpha(\tau)=0$. We also assume that $d\theta$ is integral where $d\in\mathbb{N}$ such that $(d,p)=1$. Consider the following diagram:
\begin{center}
\begin{tikzcd}
 & \mathcal{P}_{\theta}(\mathcal{O}_{z}) \arrow[d]\\
G'_{d(\theta+\tau)}(\mathcal{O}_{w}) \arrow[r] & G(\mathcal{O}_{w})
\end{tikzcd}
\end{center}
Taking an element $g \in  \mathcal{P}_{\theta}(\mathcal{O}_{z})$, we consider $g$ as an element in $G(\mathcal{K}_w)$ by substituting $z=w^d$. Then,  we can identify $\mathcal{P}_{\theta}(\mathcal{O}_{z})$ with a subgroup of $G(\mathcal{O}_{w})$ (actually, it is not a subgroup, they are the same under the conjugation) via
\begin{align*}
\mathcal{P}_{\theta}(\mathcal{O}_{z}) \subseteq G(\mathcal{K}_w) \rightarrow G(\mathcal{O}_{w}), \quad g \rightarrow {\rm Ad}(w^{d\theta})g.
\end{align*}
With this idea in mind, we are ready to prove the following lemma:
\begin{lem}\label{group containment}
The group $G'_{d(\theta+\tau)}(\mathcal{O}_{w})$ can be identified with the parahoric group $\mathcal{P}_{\frac{\theta+\theta_{\tau}}{p}}(\mathcal{O}'_{z})$ over $\mathbb{D}'_{z}$.
\end{lem}

\begin{proof}
The proof of this lemma is similar to that of Lemm \ref{parahoric of rational semisimple}. First, let us look at the situation at the level of lie algebras. $\mathfrak{p}_{\theta}(\mathcal{O}_{z})$ consists of elements of the form $\sum a_{i}z^i$ where $a_{i}\in \sum\limits_{\alpha}\mathfrak{g}_{\alpha}$ such that $\langle \theta, \alpha \rangle +i\geq 0$. With respect to the discussion above, if the element $\sum a_{i}z^i$ lies in the intersection of Lie algebras of $\mathcal{P}_{\theta}(\mathcal{O}_{z})$ and $G'_{d(\theta+\tau)}(\mathcal{O}_{w})$, we get the following condition on $\sum a_{i}z^{i}$:
$$\langle d\theta, \alpha \rangle +di = \langle d\theta+d\theta_{\tau},\alpha \rangle + kp$$
where $k\in\mathbb{Z}$. Since $\theta_{\tau}$ is integral and $(d,p)=1$, this implies that $d\, | \, k$. Let us write $k=md$. The condition $\langle \theta, \alpha \rangle +i\geq 0$ translates into $\langle \frac{(\theta+\theta_{\tau})}{p}, \alpha \rangle +m\geq 0$. This proves the claim at the level of lie algebras. Moreove, the condition $\langle \theta,\alpha \rangle +i=0$ translates into $\langle \frac{(\theta+\theta_{\tau})}{p}, \alpha \rangle +m=0$, this implies that the Levi quotient of the lie algebra of $\mathcal{P}_{\frac{\theta+\theta_{\tau}}{p}}$ can be identified with the Lie algebra of the centralizer of $d\theta+d\theta_{\tau}$ in the connected reductive group $Z_{G}(\zeta^{d\theta})$. Since the intersection of the Levi quotient of $\mathcal{P}_{\theta}(\mathcal{O}_{z})$ and $G'_{d(\theta+\theta_{\tau})}(\mathcal{O}_{w})$ is the centralizer of the semisimple element $d\theta+d\theta_{\tau}$ in the connected reductive group $Z_{G}(\zeta^{d\theta})$, which is connected, this implies the claim.
\end{proof}

Let ${\rm Locsys}^{tame}_{\mathcal{P}_\theta,\tau}(\mathbb{D})$ be the category of logahoric $\mathcal{P}_\theta(\mathcal{O})$-connections on $\mathbb{D}$ such that the rational semisimple part of the residue is $\tau$ (up to conjugation), and let ${\rm Higgs}^{tame}_{G'_{\theta+\tau},irr}(\mathbb{D}')$ be the category of logahoric $G'_{\theta+\tau}(\mathcal{O})$-Higgs bundles on $\mathbb{D}'$ such that the semisimple part of the residue is irrational (up to conjugation). Now we are ready to prove the \emph{local tame parahoric nonabelian Hodge correspondence}.

\begin{thm}\label{classification for parahoric connection}
The category ${\rm Locsys}^{tame}_{\mathcal{P}_\theta,\tau}(\mathbb{D})$ is equivalent to ${\rm Higgs}^{tame}_{G'_{\theta+\tau},irr}(\mathbb{D}')$. Moreover, the $p$-curvature of the logahoric $\mathcal{P}_\theta(\mathcal{O})$-connection is zero if and only if the corresponding logahoric $G'_{\theta+\tau}(\mathcal{O})$-Higgs bundle has zero Higgs field.
\end{thm}

\begin{proof}
In the proof, we will use the the correspondence between equivariant bundles and parahoric torsors, and we follow the notations $\mathbb{D}_x= \Spec(\mathcal{O}_z)$ and $\mathbb{D}_y=\Spec(\mathcal{O}_w)$ in \S\ref{sect corresp}. Let $d$ be a positive integer such that $z=w^d$, $d\theta$ is integral and $(d,p)=1$. Denote by $\Gamma$ the cyclic group of order $d$. The following diagram gives the idea of the proof.
\begin{center}
\begin{tikzcd}
{\rm Locsys}^{tame}_{\mathcal{P}_\theta,\tau}(\mathbb{D}_x )  \arrow[d,dotted] \arrow[rr,"\text{Proposition~\ref{parahoric connection equivariant}}"]  & & {\rm Locsys}^{tame}_{G, d(\theta+\tau)}([\mathbb{D}_y /\Gamma]) \arrow[d, "\text{Proposition~\ref{standard form in general}}"] \\
{\rm Higgs}^{tame}_{G'_{\theta+\tau},irr}(\mathbb{D}'_x )  & & {\rm Higgs}^{tame}_{G'_{d(\theta+\tau)}, irr}([\mathbb{D}'_y /\Gamma]) \arrow[ll, "\text{Theorem~\ref{equivalence for higgs bundles}}"]
\end{tikzcd}
\end{center}
Take a logarithmic connection in ${\rm Locsys}^{tame}_{\mathcal{P}_\theta,\tau}(\mathbb{D}_x )$, and assume that the connection is represented by $d - A\frac{dz}{z}$, where $A \in \mathfrak{p}_\theta(\disc)$. By Lemma~\ref{standard form 1 parahoric}, one may assume that $A$ is of the form $(\tau+\sigma+\sum a_{i}z^{i})\frac{dz}{z}$, where $\tau$ is rational semisimple, $\sigma$ is irrational semisimple and $[\tau, a_{i}]=ia_{i}$, $[\sigma, a_{i}]=0$ for all $i$.

By Proposition~\ref{parahoric connection equivariant}, the category ${\rm Locsys}^{tame}_{\mathcal{P}_\theta,\tau}(\mathbb{D}_x )$ is equivalent to ${\rm Locsys}^{tame}_{G, d(\theta+\tau)}([\mathbb{D}_y /\Gamma])$. We just want to remind the reader that the residue $\tau$ changes to $d(\theta+\tau)$ via the transformation $\Delta(W) = w^{d\theta}$ and the substitution $z=w^d$, which is easily observed from the proof of Proposition \ref{parahoric connection equivariant}.

Next, using the identification in Lemma~\ref{explicit correspondence bewteen groups} and Lemma~\ref{parahoric of rational semisimple}, we see that
\begin{align*}
    G'_{d(\theta+\tau)}(\mathcal{O}_w) \cong \mathcal{P}_{ \frac{d(\theta + \theta_\tau)}{p} }(\mathcal{O}'_w)
\end{align*}
over $\mathcal{O}'_w$. By Proposition~\ref{standard form in general}, the category ${\rm Locsys}^{tame}_{G, d(\theta+\tau)}([\mathbb{D}_y /\Gamma])$ is equivalent to the category ${\rm Higgs}^{tame}_{G'_{d(\theta+\tau)}, irr}([\mathbb{D}'_y /\Gamma])$.

For the last step, we have
\begin{align*}
    G'_{d(\theta+\tau)}(\mathcal{O}_w) \cong \mathcal{P}_{ \frac{\theta + \theta_\tau}{p} }(\mathcal{O}'_z) \cong G'_{\theta + \tau}(\mathcal{O}_z)
\end{align*}
by Proposition \ref{group containment}. Therefore, we apply Theorem \ref{equivalence for higgs bundles} and obtain the equivalence of categories between ${\rm Higgs}^{tame}_{G'_{\theta+\tau},irr}(\mathbb{D}'_x )$ and ${\rm Higgs}^{tame}_{G'_{d(\theta+\tau)}, irr}([\mathbb{D}'_y /\Gamma])$. This finishes the proof of this theorem.
\end{proof}

\section{Tame Parahoric Nonabelian Hodge Correspondence}\label{sect global}

In this section, we establish the global version of tame parahoric nonabelian Hodge correspondence on curves in positive characteristic. With the help of the local study in \S\ref{sect local}, we mostly follow Chen--Zhu's approach to give the correspondence (Theorem \ref{first structure theorem for G} and \ref{parahoric non hod corr}). Furthermore, we apply the local results in \S\ref{sect local} and give a more precise description of the stack for Higgs bundles in the correspondence (Proposition \ref{structure of X} and \ref{structure of parahoric X}).

\subsection{Artin-Schreier Map}
Let $k[\mathfrak{g}]$ and $k[\mathfrak{t}]$ be the algebras of regular functions on $\mathfrak{g}$ and $\mathfrak{t}$. Chevalley restriction theorem shows that we have an isomorphism $k[\mathfrak{g}]^G \cong k[\mathfrak{t}]^W$, where $W$ is the Wely group. Let $\mathfrak{c}={\rm Spec}(k[\mathfrak{t}]^W)$, and denote by $p: \mathfrak{t} \rightarrow \mathfrak{c}$ the projection. Let $\mathfrak{t}^{irr}$ be the set of irrational elements in $\mathfrak{t}$. A semisimple conjugacy class in $\mathfrak{g}$ is \emph{irrational} if it is the conjugacy class of an element $x\in\mathfrak{t}^{irr}$. Let $\mathfrak{c}^{irr}$ be the set of irrational semisimple conjugacy class. It is easy to see that $\mathfrak{c}^{irr}$ is an open subset of $\mathfrak{c}$.

\begin{lem}\label{irrational locus}
Let $\mathfrak{t}'$ (resp. $\mathfrak{c}'$) be the Frobenious twist of $\mathfrak{t}$ (resp. $\mathfrak{c}$). Then the following diagram is Cartesian:
$$
\xymatrix{
\mathfrak{t}^{irr} \ar[r]^{AS} \ar[d]^{p} & \mathfrak{t}' \ar[d]^{p'}\\
\mathfrak{c}^{irr} \ar[r] & \mathfrak{c}'
}
$$
where $AS$ stands for the Artin-Schreier map. Moreover, the morphism $\mathfrak{c}^{irr}\rightarrow\mathfrak{c}'$ is surjective.
\end{lem}
\begin{proof}
First, we have a $W$-equivariant commutative diagram:
$$
\xymatrix{
\mathfrak{t} \ar[r]^{AS} \ar[d]^{p} & \mathfrak{t}' \ar[d]^{p'}\\
\mathfrak{c} \ar[r] & \mathfrak{c}'
}
$$
It remains to show that when restricted to $\mathfrak{c}^{irr}$, the fibers of $p$ and $p'$ can be canonically identified. Indeed, if $\alpha( x)=\lambda$, then $\alpha(AS(x))=\lambda^p-\lambda$. If $x\in\mathfrak{t}^{irr}$, then $\lambda^p-\lambda=0$ if and only if $\lambda=0$. Hence we have $w(x)=x$ if and only if $w(AS(x))=AS(x)$ for all $w\in W$. This proves the claim.
\end{proof}

By Chevalley restriction theorem, we have a natural map $\chi: \mathfrak{g} \rightarrow \mathfrak{c}$ induced by $k[\mathfrak{c}] \rightarrow k[\mathfrak{g}]$. Denote by $kos: \mathfrak{c} \rightarrow \mathfrak{g}$ the Kostant section. We define the group scheme $I$ over $\mathfrak{g}$ as
\begin{align*}
I=\{(g,x) \in G \times \mathfrak{g} \text{ }|\text{ } {\rm Ad}_g(x)=x\}.
\end{align*}
Then, define $J=kos^* I$. There is a canonical isomorphism $\chi^* J|_{\mathfrak{g}^{reg}} \cong I|_{\mathfrak{g}^{reg}}$, where $\mathfrak{g}^{reg}$ is the open subset of regular elements, and the morphisms $I \rightarrow \mathfrak{g}$ and $J \rightarrow \mathfrak{c}$ are $\mathbb{G}_m$-equivariant. Furthermore, we have a tautological section $\tau: \mathfrak{c} \rightarrow J$, which is also $\mathbb{G}_m$-equivariant (see \cite[\S 2.3]{Nonabelian Hodge prime char}).

\begin{cor}
Let $J'$ be the regular centralize group scheme over $\mathfrak{c}'$ under the Frobenius twist. The pullback of ${\rm Lie}(J')$ is canonically isomorphic to ${\rm Lie}(J)$ over the open set $\mathfrak{c}^{irr}$, i.e.
\begin{align*}
Fr^*({\rm Lie}(J'))|_{\mathfrak{c}^{irr}} \cong {\rm Lie}(J)|_{\mathfrak{c}^{irr}}.
\end{align*}
\end{cor}
\begin{proof}
This follows from Lemma~\ref{irrational locus} and \cite[Proposition 12.5]{The gerbe of Higgs bundles}.
\end{proof}

\subsection{$\boldsymbol{p}$-Hitchin Map for Tame Local Systems}
Recall that we use the notation $\mathscr{L}:=\Omega_X(D)$ for the logarithmic cotangent sheaf. Regarding the line bundle $\mathscr{L}$ as a $\mathbb{G}_m$-torsor, we can twist every Lie algebras and group schemes by $\mathscr{L}$. For example, denote by $\mathfrak{g}_{\mathscr{L}}:=\mathfrak{g} \times_{\mathbb{G}_m} \mathscr{L}^{\times}$. With respect to the data above, the stack of logarithmic $G$-Higgs bundles ($\mathscr{L}$-twisted Higgs bundles) over $X$ can be regarded as the stack of sections
\begin{align*}
{\rm Higgs}^{tame}_{G, \mathscr{L}} = {\rm Sect}(X,[\mathfrak{g}_\mathscr{L}/G]),
\end{align*}
where we add the subscript $\mathscr{L}$ to emphasize that it is $\mathscr{L}$-twisted. Furthermore, the \emph{Hitchin base} $B_\mathscr{L}:={\rm Sect}(X,\mathfrak{c}_\mathscr{L})$ is regarded as the stack of sections, which is also a scheme. The morphism $\chi: \mathfrak{g} \rightarrow \mathfrak{c}$ induces a natural map $[\chi_\mathscr{L}/G]: [\mathfrak{g}_\mathscr{L}/G] \rightarrow \mathfrak{c}$, and then we have
\begin{align*}
h : {\rm Higgs}^{tame}_{G,\mathscr{L}}={\rm Sect}(X,[\mathfrak{g}_\mathscr{L}/G]) \rightarrow {\rm Sect}(X,\mathfrak{c}_\mathscr{L})=B_\mathscr{L},
\end{align*}
which is the \emph{Hitchin map}.

Let $(E,\phi) \in {\rm Higgs}^{tame}_{G,\mathscr{L}}$ be a logarithmic $G$-Higgs bundle, and denote by $h_{E,\phi}: X \rightarrow [\mathfrak{g}_\mathscr{L}/G]$ the corresponding section with image $b: X \rightarrow \mathfrak{c}_\mathscr{L}$ in the Hitchin base $B_\mathscr{L}$. Taking pullback of the following diagram
\begin{center}
\begin{tikzcd}
b^* J_\mathscr{L}  \arrow[d, dotted] \arrow[r, dotted]  & J_\mathscr{L} \arrow[d] \\
X \arrow[r, "b"] & \mathfrak{c}_\mathscr{L}
\end{tikzcd}
\end{center}
we get a smooth group scheme $J_b:=b^* J_\mathscr{L}$ over $X$. On the other hand, the morphism $\chi^* J \rightarrow I$ induces the morphism $[\chi_\mathscr{L} /G]^* J_\mathscr{L} \rightarrow [I_\mathscr{L}/G]$ of group schemes over $[\mathfrak{g}_\mathscr{L}/G]$. Pulling back the morphism to $X$ via $h_{E,\phi}$, we get
\begin{align*}
a_{E,\phi}: J_b \rightarrow h_{E,\phi}^*[I_\mathscr{L}/G]= {\rm Aut}(E,\phi) \subseteq {\rm Aut}(E).
\end{align*}
Thus, we can twist $(E,\phi) \in h^{-1}(b)$ by a $J_b$-torsor.

Recall that the stack of tame $G$-local systems ${\rm Locsys}^{tame}_{G}$ parametrizes pairs $(E,\nabla)$, where $E$ is a $G$-torsor and $\nabla$ is a logarithmic $G$-connection. In the logarithmic case, one still has an analogue notion of the $p$-curvature as in \cite[A.6]{Nonabelian Hodge prime char}. Let $X'$ be the Frobenius twist of $X$ with natural morphism $Fr: X \rightarrow X'$. Define $D'=Fr(D)$ and $\mathscr{L}':=\Omega_{X'}(D')$. Clearly, $Fr^* \mathscr{L}' \cong \mathscr{L}^p$, and the $p$-curvature $\Psi(\nabla)$ of a logarithmic $G$-connection $\nabla$ is a horizontal section of ${\rm Ad}(E)\otimes \mathscr{L}^p$. Then, there is a unique morphism
\begin{align*}
    h_p: {\rm Locsys}^{tame}_{G} \rightarrow B_{\mathscr{L}'},
\end{align*}
which is called the \emph{$p$-Hitchin map}, such that the following diagram commutes
\begin{center}
\begin{tikzcd}
{\rm Locsys}^{tame}_{G}  \arrow[d] \arrow[r,"h_p"]  & B_{\mathscr{L}'} \arrow[d] \\
{\rm Higgs}^{tame}_{G,\mathscr{L}^p} \arrow[r] & B_{\mathscr{L}^p}
\end{tikzcd}
\end{center}
where ${\rm Higgs}^{tame}_{G,\mathscr{L}^p}:={\rm Sect}(X, [\mathfrak{g}_{\mathscr{L}^p}/G])$ and $B_{\mathscr{L}^p}:={\rm Sect}(X,\mathfrak{c}_{ \mathscr{L}^p })$ \cite[Proposition 3.1, Lemma 3.2]{Nonabelian Hodge prime char}. Now let $(E,\nabla)$ be a tame $G$-local system. Denote by $\Psi:=\Psi(\nabla)$ the $p$-curvature and set $b'=h_{p}(E,\nabla) \in B_{\mathscr{L}'}$ with $b^p$ the image in $B_{\mathscr{L}^p}$. Then, we obtain a natural morphism
\begin{align*}
a_{E,\Psi} : J_{b^p} \rightarrow {\rm Aut}(E).
\end{align*}
Moreover, we have $Fr^* (J'_{b'}) =J_{b^p}$, and there is a canonical connection $\nabla$ on $J_{b^p}$ such that $(J_{b^p})^\nabla = J'_{b'}$. Based on the above discussion, We have the following lemma.

\begin{lem}\label{horizontal regular centralizer}
The homomorphism $a_{E,\Psi}$ is horizontal.
\end{lem}

\begin{proof}
It is enough to show that the restriction of $a_{E,\Psi}$ to $X \backslash D$ is horizontal, which is proven in \cite[Lemma 3.3]{Nonabelian Hodge prime char}.
\end{proof}

The group schemes $J_{b^p}$ can be realized a family of group schemes $J^p$ over $X \times B_{\mathscr{L}'}$, which is equipped with a natural connection along $X$. From Lemma~\ref{horizontal regular centralizer}, we conclude that for any tame $J^{p}$-local system $(P,\nabla_P)$ and a tame $G$-local system $(E,\nabla)$, one may apply \cite[A.5]{Nonabelian Hodge prime char} to get a tame $G$-local system:
\begin{align*}
((P,\nabla_P),(E,\nabla))\mapsto P\otimes E :=(a_{E,\Psi})_{*}P\otimes E.
\end{align*}
This actually defines an action:
\begin{equation}\label{action on locsys}
{\rm Locsys}^{tame}_{J^p} \times {\rm Locsys}^{tame}_{G} \rightarrow   {\rm Locsys}^{tame}_{G}.
\end{equation}

\begin{lem}\label{action over phitchin}
Given a tame $G$-local system $(E,\nabla)$, let $b'=h_{p}(E,\nabla)$. If $(P,\nabla_P)$ is a tame $J^{p}$-local system such that its $p$-curvature is zero, then $h_{p}(P\otimes E,\nabla_{P\otimes E})=b'$.
\end{lem}

\begin{proof}
Similar to the proof of Lemma \ref{horizontal regular centralizer}, one only need to check this over $X \backslash D$, in which case one may apply \cite[Lemma 3.4]{Nonabelian Hodge prime char}.
\end{proof}

Now we consider an analogue of all the constructions above for the parahoric case. Let $\boldsymbol\theta$ be a collection of tame weights and let $\mathcal{P}_{\boldsymbol\theta}$ be the parahoric group scheme over $X$ corresponding to $\boldsymbol\theta$. Let ${\rm Higgs}^{tame}_{\mathcal{P}_{\boldsymbol\theta},\mathscr{L}}$ be the stack of logarhoric $\mathcal{P}_{\boldsymbol\theta}$-Higgs bundles on $X$. First we have the following:
\begin{lem}\label{parahoric hitchin fibration}
There exists a parahoric Hitchin fibration $h$: ${\rm Higgs}^{tame}_{\mathcal{P}_{\boldsymbol\theta},\mathscr{L}}\rightarrow {\rm Sect}(X,\mathfrak{c}_{\mathscr{L}})$.
\end{lem}

\begin{proof}
The parahoric version of the Hitchin morphism has been considered in \cite[\S 4]{Global Springer}, and we give a slightly different proof in the view point of equivariant Higgs bundles. Let $(E,\phi)$ be a parahoric $\mathcal{P}_{\boldsymbol\theta}$-Higgs bundle. Away from the support of $D$, the structure group of $E$ is identified with $G$. Then, we have a morphism ${\rm Ad}(E) \otimes\mathscr{L}\rightarrow\mathfrak{c}_{\mathscr{L}}$ away from the support of $D$. It remains to show that this morphism can be extended to $X$. Then, we can work locally on $X$ and assume that $\mathscr{L}$ is trivial and there exists a cover $Y$ over $X$ of degree $d$ such that $(E, \phi)$ can be identified with a logarithmic $(\Gamma,G)$-Higgs bundle on $Y$ \S\ref{subsect Higgs}. If we identify $\phi$ with an element in ${\rm Sect}(Y,\mathfrak{g}_{\mathscr{L}_Y})$, where $\mathscr{L}_Y$ is the pullback of $\mathscr{L}$ to $Y$, then we get an element in ${\rm Sect}^{\Gamma}(Y,\mathfrak{c}_{\mathscr{L}_Y})$, which is the same as ${\rm Sect}(X,\mathfrak{c}_{\mathscr{L}})$. This finishes the proof of this lemma.
\end{proof}

Given a logahoric $\mathcal{P}_{\boldsymbol\theta}$-Higgs bundle $(E,\phi)$ on $X$, let $b$ be its image in ${\rm Sect}(X,\mathfrak{c}_{\mathscr{L}})$. Let $J_{b}$ be the regular centralizer group scheme over $X$.

\begin{lem}\label{regular centralizer map for parahoric groups}
Let $(E,\phi)$ be a logahoric $\mathcal{P}_{\boldsymbol\theta}$-Higgs bundle and let $b=h(E,\phi)$. Let $J_{b}$ be the regular centralizer group scheme on $X$. Then we have a group homomorphism $J_{b}\rightarrow {\rm Aut}(E)$ over $X$.
\end{lem}

\begin{proof}
Away from the support of $D$, $(E,\phi)$ is identified with a $G$-Higgs bundle, and thus we get a morphism $J_{b}\rightarrow {\rm Aut}(E)$ on $X \backslash D$. We claim that this morphism extends to $X$. The approach is exactly the same as Lemma \ref{parahoric hitchin fibration}, and we only need to look at formal neighborhood of each point in the support of $D$. Let us assume $X={\rm Spec}(k[[z]])$ and $Y={\rm Spec}(k[[w]])$ where $w^d=z$. The element $b$ can be viewed as either a point in ${\rm Sect}(X,\mathfrak{c}_{\mathscr{L}})$ or ${\rm Sect}(Y,\mathfrak{c}_{\mathscr{L}_Y})$, so we get regular centralizer group schemes $J_{X,b}$ and $J_{Y,b}$ over $X$ and $Y$ respectively. We have $J_{Y,b}\simeq J_{X,b}\times_{X} Y$. The Higgs bundle $(E,\phi)$ can be identified with a $(\Gamma,G)$-equivariant Higgs bundle on $Y$ as discussed in \S\ref{subsect Higgs}, so we get a morphism $J_{Y,b}\rightarrow {\rm Aut}(E)$. Passing to $\Gamma$-invariant sections, we get $J_{X,b}\rightarrow {\rm Aut}(E)$, where $J_b = J_{X,b}$.
\end{proof}

Next we look at tame parahoric local systems. Recall that ${\rm Locsys}^{tame}_{\mathcal{P}_{\boldsymbol\theta}}$ is the stack of tame $\mathcal{P}_{\boldsymbol\theta}$-local systems on $X$.

\begin{lem}\label{image of horizontal section}
Let $(E,\nabla)$ be a tame $\mathcal{P}_{\boldsymbol\theta}$-local system on $X$. The $p$-curvature of $(E,\nabla)$ is actually in ${\rm Sect}(X',\mathfrak{c}_{\mathscr{L}'})$, and then we obtain a natural morphism
\begin{align*}
    h_p : {\rm Locsys}^{tame}_{\mathcal{P}_{\boldsymbol\theta}} \rightarrow {\rm Sect}(X',\mathfrak{c}_{\mathscr{L}'}),
\end{align*}
which is called the $p$-Hitchin morphism. More generally, if $\Psi$ is a horizontal section of ${\rm Ad}(E) \otimes F^*(\mathscr{L}')$, then the image of $\Psi$ in ${\rm Sect}(X,\mathfrak{c}_{Fr^{*}(\mathscr{L}')})$ is also horizontal.
\end{lem}

\begin{proof}
Away from the support of $D$, $\mathcal{P}_{\boldsymbol\theta}$ can be identified with $G$, so the arguments in \cite[Proposition 3.1 and Lemma 3.2]{Nonabelian Hodge prime char} shows that $h(E,\Psi)$ is a horizontal section of ${\rm Ad}(E) \otimes F^*(\mathscr{L}')$ with respect to the canonical connection on $F^{*}\mathscr{L}'$. With the same proof as in Lemma \ref{parahoric hitchin fibration}, we can extend it to $X$ and and obtain a horizontal section in ${\rm Sect}(X,\mathfrak{c}_{Fr^*(\mathscr{L}')})$.
\end{proof}

Let $(E,\nabla)$ be a tame $\mathcal{P}_{\boldsymbol\theta}$-local system, denote by $b'$ its image in ${\rm Sect}(X',\mathfrak{c}_{\mathscr{L}'})$. Similar to the case of principal bundles, the element $b'$ defines a regular centralizer group scheme $J'_{b'}$ over $X'$, and let $J_{b^p}$ be its pullback to $X$. Then one gets a group homomorphism $a_{E,\Psi}: J_{b^p}\rightarrow {\rm Aut}(E)$, where we use the same notation. Note that $J_{b^p}$ is equipped with a canonical connection as the pullback of $J'_{b'}$, while ${\rm Aut}(E)$ is equipped with a logarithmic connection. Then, we obtain the following lemma as an analogue of Lemma \ref{horizontal regular centralizer}.

\begin{lem}\label{horizontal regular centralizer parh}
The homomorphism $a_{E,\Psi}$ is horizontal.
\end{lem}

Now as we did for tame $G$-local systems, we still have an action of tame $J_{b^p}$-local systems on tame parahoric local systems:
\begin{equation}\label{action on locsys}
{\rm Locsys}^{tame}_{J^p} \times {\rm Locsys}^{tame}_{\mathcal{P}_{\boldsymbol\theta}} \rightarrow   {\rm Locsys}^{tame}_{\mathcal{P}_{\boldsymbol\theta}}.
\end{equation}
Lemma ~\ref{action over phitchin} still holds in the parahoric case. Moreover, we define two substacks $\mathcal{A}_0 \subseteq \mathcal{A} \subseteq {\rm Locsys}^{tame}_{J^p}$.
\begin{itemize}
\item $\mathcal{A}$ is the stack of tame $J^p$-local systems with zero $p$-curvature.
\item $\mathcal{A}_{0} \subseteq \mathcal{A}$ is the substack of $(P,\nabla_P) \in \mathcal{A}$, of which $\nabla_P$ has no poles.
\end{itemize}
Clearly, $\mathcal{A}$ and $\mathcal{A}_{0}$ are group stacks over $B_{\mathscr{L}'}$. Lemma \ref{action over phitchin} implies the following:
\begin{cor}
There exists a natural action of $\mathcal{A}$ on ${\rm Locsys}^{tame}_{G}$ which preserves the p-Hitchin map for ${\rm Locsys}^{tame}_{G}$.
\end{cor}

\subsection{Vector Bundle $\mathscr{B}_{\mathscr{L}'}$}
Now we construct a vector bundle $\mathscr{B}_{\mathscr{L}'}$ over $B_{\mathscr{L}'}$, of which the fiber is $H^0(X', {\rm Lie}(J'_{b'}) \otimes \mathscr{L}')$ for each $b' \in B_{\mathscr{L}'}$.
The fiber is actually the space of $p$-curvatures for logarithmic $J'_{b'}$-connections. Then the stack ${\rm Locsys}^{tame}_{J^p}$ is equipped with a natural morphism to $\mathscr{B}_{\mathscr{L}'}$. In this subsection, we will prove that this morphism is smooth.

Let $\mathcal{G}'$ be a smooth affine commutative group scheme over $X'$ and denote by $\mathcal{G}$ its Frobenius pullback to $X$. The following sequence is an analogue of \cite[Proposition A.7]{Nonabelian Hodge prime char} in our setting:
\begin{lem}\label{four term sequence}
On $X'_{\rm et}$, we have a sequence of sheaves
\begin{align*}
0 \rightarrow \mathcal{G}' \rightarrow Fr_*\mathcal{G}\xrightarrow{Fr_* d {\rm log}} {\rm Lie}(\mathcal{G}')\otimes Fr_* \mathscr{L} \xrightarrow{h_{p}} {\rm Lie}(\mathcal{G}') \otimes \mathscr{L} \rightarrow 0.
\end{align*}
It is exact except at the position ${\rm Lie}(\mathcal{G}')\otimes Fr_* \mathscr{L}$.
\end{lem}
\begin{proof}
By Proposition A.7 of \cite{Nonabelian Hodge prime char}, we have the following sequence:
\begin{equation}\label{exact sequence}
0\rightarrow\mathcal{G}'\rightarrow Fr_*\mathcal{G}\xrightarrow{Fr_* d {\rm log}} {\rm Lie}(\mathcal{G}') \otimes Fr_*\Omega_{X}\xrightarrow{h_{p}} {\rm Lie}(\mathcal{G}') \otimes \Omega_{X'}\rightarrow 0,
\end{equation}
which is exact. Recall that $\mathscr{L}=\Omega_X(D)$. So it remains to show $h_{p}$ is surjective as a morphism from ${\rm Lie}(\mathcal{G}')\otimes Fr_*\Omega_{X}(x)$ to ${\rm Lie}(\mathcal{G}')\otimes\Omega_{X'}(x')$ for each $x \in D$. Using the surjectivity of $h_{p}$ in \ref{exact sequence}, we only need to show the restriction of $h_{p}$ to $x$ is surjective. By the description of $h_{p}$ in \cite[Lemma A.8]{Nonabelian Hodge prime char}, we have $h_{p}(\eta_{x})=\eta^{[p]}_{x'}-\eta_{x'}$, where $\eta_{x'}\in{\rm Lie}(\mathcal{G}')|_{x'}$ and $(-)^{[p]}$ is the $p$-power operator. Since $\mathcal{G}$ is a smooth affine commutative group scheme, we may assume $\mathcal{G}'|_{x'}$ is either $\mathbb{G}_{m}$ or $\mathbb{G}_{a}$ and the assertion follows directly.
\end{proof}

Now we assume that $\mathcal{G}'$ is a regular centralizer group scheme (see \cite{The gerbe of Higgs bundles}).
\begin{lem}\label{surjective for residual}
Let $\eta$ be an element in ${\rm Lie}(\mathcal{G}')|_{x'}$ such that $\eta^{[p]}-\eta=0$. Then, in the \'etale topology, there exists $\omega\in {\rm Lie}(\mathcal{G}')\otimes Fr_* \Omega_{X}(x)$ such that $h_{p}(\omega)=0$ and the residue of $\omega$ at $x$ is equal to $\eta$.
\end{lem}

\begin{proof}
One may assume $X'$ is affine and that $\mathcal{G}'$ is the regular centralizer determined by an element $\phi\in\mathfrak{g}(X')$. Let $L$ be the Levi subgroup of $G:=\mathcal{G}'|_{x'}$ determined by the semisimple part of the fiber $\phi_{x'}$. Then $\eta$ can be identified with an integral element in $Z(\mathfrak{l})$. By \cite[Lemma B.0.3]{autoduality of hitchin}, one may pass to an \'etale cover of $X'$ so that $\mathcal{G}'$ is the regular centralizer for the Levi subgroup $L$. In this case, $Z(L)$ can be identified with a subgroup of $\mathcal{G}'$, hence locally there exists an element $f\in\mathcal{G}'(X'\backslash x')$ such that the residue of $\omega=Fr_*d{\rm log}(f)$ is $\eta$ and $h_{p}(\omega)=0$.
\end{proof}

\begin{cor}\label{gerbe of fix p curvature}
Let $q\in{\rm Lie}(\mathcal{G}')\otimes\Omega_{X'}(x')$. We take an element $\eta\in{\rm Lie}(\mathcal{G}')|_{x'}$ such that $h_{p}(\eta)$ is equal to the residue of $q$ at $x'$. Then the category of tame $\mathcal{G}$-local systems on $X$ such that the $p$-curvature equal to $q$ and the residue equal to $\eta$ is a gerbe over ${\rm Bun}_{\mathcal{G}'}$.
\end{cor}

\begin{proof}
Combine Lemma~\ref{four term sequence} and Lemma~\ref{surjective for residual}, one concludes that locally over $X'$, there exists $\omega\in{\rm Lie}(\mathcal{G}')\otimes Fr_*\Omega_{X}(x)$ such that $h_{p}(\omega)=q$ and that the residue of $\omega$ is equal to $\eta$. Moreover, two different choices of such $\omega$ satisfies the following conditions: $\omega_{1}-\omega_{2}\in{\rm Lie}(\mathcal{G}')\otimes Fr_*\Omega_{X}$ and $h_{p}(\omega_{1}-\omega_{2})=0$. Now the claim follows from the exactness of \ref{exact sequence}.
\end{proof}

Now we are ready to prove the following:
\begin{lem}\label{smoothness of locsys}
${\rm Locsys}^{tame}_{J^p}$ is smooth over $\mathscr{B}_{\mathscr{L}'}$.
\end{lem}

\begin{proof}
Fixing a point $b' \in B_{\mathscr{L}'}$, the fibers of ${\rm Locsys}^{tame}_{J^p}$ and $\mathscr{B}_{\mathscr{L}'}$ over $b$ are smooth. We will show that
\begin{align*}
{\rm Locsys}^{tame}_{J_{b^p}} \rightarrow H^0(X', {\rm Lie}(J'_{b'}) \otimes \mathscr{L}')
\end{align*}
is smooth. It suffices to prove that the induced morphism on tangent space is surjective. Let $K^{\bullet}$ be the complex given by:
\begin{align*}
K^{\bullet}: {\rm Lie}(J'_{b'})\xrightarrow{\nabla}{\rm Lie}(J'_{b'}) \otimes \mathscr{L}'.
\end{align*}
The tangent space of ${\rm Locsys}^{tame}_{J_{b^p}}$ is given by $H^{1}(X',K^{\bullet})$. Let $\mathcal{H}^{i}(K^{\bullet})$ be the cohomology group of $K^{\bullet}$ at degree $i$. Then by Lemma~\ref{four term sequence}, the cohomology $\mathcal{H}^{1}(K^{\bullet})$ admits a surjective morphism to ${\rm Lie}(J'_{b'})\otimes\mathscr{L}'$ such that the kernel is supported at $x'$. Thus, one may view $\mathcal{H}^{1}(K^{\bullet})$ as a coherent sheaf on $X'$ and that ${\rm Lie}(J'_{b'}) \otimes \mathscr{L}'$ is a direct summand of $\mathcal{H}^{1}(K^{\bullet})$. Then the morphism
\begin{align*}
H^0(X',\mathcal{H}^{1}(K^{\bullet}))\rightarrow H^{0}(X',{\rm Lie}(J'_{b'}) \otimes \mathscr{L}' )
\end{align*}
is surjective. Since $H^{2}(X',\mathcal{H}^{0}(K^{\bullet}))=0$, we have $H^{1}(X',K^{\bullet})\rightarrow H^0(X',\mathcal{H}^{1}(K^{\bullet}))$ is surjective. Therefore, we have
\begin{align*}
H^{1}(X',K^{\bullet})\rightarrow H^{0}(X',{\rm Lie}(J'_{b'}) \otimes \mathscr{L}')
\end{align*}
is surjective. This finishes the proof.
\end{proof}


\subsection{Tame Nonabelian Hodge Correspondence}\label{subsect_tnhc}
By Lemma \ref{smoothness of locsys}, we have a smooth morphism ${\rm Locsys}^{tame}_{J^p} \rightarrow \mathscr{B}_{\mathscr{L}'}$. The tautological section $\mathfrak{c} \rightarrow {\rm Lie}(J)$ induces a canonical section $\tau':B_{\mathscr{L}'} \rightarrow \mathscr{B}_{\mathscr{L}'}$.
Then we define $\mathcal{H}$ as the pullback of the following diagram
\begin{center}
\begin{tikzcd}
\mathcal{H}  \arrow[d,dotted] \arrow[r,dotted]  & {\rm Locsys}^{tame}_{J^p} \arrow[d] \\
B_{\mathscr{L}'} \arrow[r, "\tau'"] & \mathscr{B}_{\mathscr{L}'} \ ,
\end{tikzcd}
\end{center}
\noindent For each point $b' \in B_{\mathscr{L}'}$, the fiber of $\mathscr{B}_{\mathscr{L}'} \rightarrow B_{\mathscr{L}'}$ is $H^0(X',{\rm Lie}(J'_{b'})\otimes \mathscr{L}')$. One may consider the residue morphism:
\begin{align*}
H^0(X',{\rm Lie}(J'_{b'})\otimes
\mathscr{L}')\xrightarrow{\residual} {\rm Lie}(J'_{b'})\otimes
\mathscr{L}'|_{x'}.
\end{align*}
By Lemma \ref{smoothness of locsys}, we get a morphism by taking over all $b' \in B_{\mathscr{L}'}$
\begin{align*}
{\rm Locsys}^{tame}_{J^{p}}\rightarrow {\rm Lie}(J')\otimes
\mathscr{L}'|_{x'}.
\end{align*}
Similarly, by taking the residue of logarithmic connections, we get a morphism
\begin{align*}
{\rm Locsys}^{tame}_{J^{p}}\rightarrow {\rm Lie}(J')\otimes
\mathscr{L}|_{x}.
\end{align*}
Then under the Artin-Schreier map, we have the following commutative diagram:
$$
\xymatrix{
{\rm Locsys}^{tame}_{J^{p}} \ar[r] \ar[rd] & {\rm Lie}(J')\otimes
\mathscr{L}|_{x} \ar[d]^{\ppower}\\
 & {\rm Lie}(J')\otimes
\mathscr{L}'|_{x'}
 }
$$
where ${\rm AS}$ is induced from the map
\begin{align*}
{\rm Lie}(J') \rightarrow {\rm Lie}(J'), \quad g \rightarrow g^{[p]}-g,
\end{align*}
where $(-)^{[p]}$ is the $p$-power operator. With respect to the above discussion, we get the following picture:
\begin{equation}\label{key diagram for hitchin base}
\xymatrix{
\mathcal{H} \ar[r] \ar[d] & {\rm Locsys}^{tame}_{J^p} \ar[r] \ar[d] & {\rm Lie}(J')\otimes
\mathscr{L}|_{x} \ar[d]^{\ppower}\\
B_{\mathscr{L}'} \ar[r]^{\tau'} & \mathscr{B}_{\mathscr{L}'} \ar[r] & {\rm Lie}(J')\otimes
\mathscr{L}'|_{x'}
}
\end{equation}
Here we view ${\rm Lie}(J')\otimes \mathscr{L}|_{x}$ and ${\rm Lie}(J')\otimes \mathscr{L}'|_{x'}$ as vector bundles over $\mathfrak{c}_{\mathscr{L}}|_{x}$. The diagram \ref{key diagram for hitchin base} implies that we have the induced morphism
\begin{align*}
\mathcal{H}\rightarrow B_{\mathscr{L}'} \times_{ {\rm Lie}(J')\otimes
\mathscr{L}'|_{x'} } {\rm Lie}(J')\otimes
\mathscr{L}|_{x}.
\end{align*}
To simplify the notations, let us denote
\begin{align*}
B^{\rm ext}_{\mathscr{L}'}:=B_{\mathscr{L}'} \times_{ {\rm Lie}(J')\otimes
\mathscr{L}'|_{x'} } {\rm Lie}(J')\otimes
\mathscr{L}|_{x}.
\end{align*}
Since $\ppower$ in \ref{key diagram for hitchin base} is \'etale and surjective, the morphism $B^{\rm ext}_{\mathscr{L}'} \rightarrow B_{\mathscr{L}'}$ is also \'etale and surjective.
\begin{lem}\label{torsors over group stack}
The stack $\mathcal{H}$ is an $\mathcal{A}$-torsor  over $B_{\mathscr{L}'}$, and it is also an $(\mathcal{A}_{0} \times_{B_{\mathscr{L}'}} B^{\rm ext}_{\mathscr{L}'})$-torsor over $B^{\rm ext}_{\mathscr{L}'}$.
\end{lem}

\begin{proof}
First, the fibers of $\mathcal{H} \rightarrow B^{\rm ext}_{\mathscr{L}'}$ (resp. $\mathcal{H} \rightarrow B_{\mathscr{L}'}$) are torsors over fibers of $\mathcal{A}_{0}\times_{B_{\mathscr{L}'}} B^{\rm ext}_{\mathscr{L}'} \rightarrow B^{\rm ext}_{\mathscr{L}'}$ (resp. $\mathcal{A} \rightarrow B_{\mathscr{L}'}$). Then, since $B^{\rm ext}_{\mathscr{L}'}$ is \'etale over $B_{\mathscr{L}'}$, $\mathcal{H}$ is smooth over both $B^{\rm ext}_{\mathscr{L}'}$ and $B_{\mathscr{L}'}$ by Lemma~\ref{smoothness of locsys}. Also, $B^{\rm ext}_{\mathscr{L}'}$ maps surjectively to $B_{\mathscr{L}'}$. Thus, it is enough to show $\mathcal{H}\rightarrow B^{\rm ext}_{\mathscr{L}'}$ is surjective, which is a consequence of Corollary~\ref{gerbe of fix p curvature} as well as \cite[Lemma 3.15 and Lemma 3.16]{Nonabelian Hodge prime char}.
\end{proof}

Let $\mathcal{X}$ be the stack that parameterizes triples $(E,\nabla,\Psi)$ such that
\begin{itemize}
\item $(E,\nabla)$ is a tame $G$-local systems with zero $p$-curvature,
\item $\Psi$ is a horizontal section of ${\rm Ad}(E)\otimes Fr^*\mathscr{L}'$
\end{itemize}
Clearly, $\mathcal{X}$ is an algebraic stack over $B_{\mathscr{L}'}$.

\begin{thm}\label{first structure theorem for G}
There exists a canonical isomorphism of stacks over $B_{\mathscr{L}'}$:
\begin{align*}
\mathcal{H}\times^{\mathcal{A}}\mathcal{X}\rightarrow {\rm Locsys}^{tame}_{G}.
\end{align*}
Similarly, there exists a canonical isomorphism of stacks over $B^{\rm ext}_{\mathscr{L}'}$:
\begin{align*}
\mathcal{H} \times^{\mathcal{A}^{\rm ext}_{0}} \mathcal{X}^{\rm ext} \rightarrow {\rm Locsys}^{tame}_{G} \times_{B_{\mathscr{L}'}} B^{\rm ext}_{\mathscr{L}'},
\end{align*}
where $\mathcal{A}^{\rm ext}_{0}$ and $\mathcal{X}^{\rm ext}$ stands for the base change of $\mathcal{A}_{0}$ and $\mathcal{X}$ to $B^{\rm ext}_{\mathscr{L}'}$.
\end{thm}

\begin{proof}
The morphism $\mathcal{H}\times^{\mathcal{A}}\mathcal{X}\rightarrow {\rm Locsys}^{tame}_{G}$ is induced by the action of ${\rm Locsys}^{tame}_{J^p}$ on ${\rm Locsys}^{tame}_{G}$ as in \eqref{action on locsys}. The inverse morphism is given by:
\begin{align*}
\mathcal{H}(-\tau')\times^{\mathcal{A}} {\rm Locsys}^{tame}_{G}  \rightarrow \mathcal{X}.
\end{align*}
where $\mathcal{H}(-\tau')$ is the pullback of ${\rm Locsys}^{tame}_{J^p}$ to $B_{\mathscr{L}'}$ via the section $B_{\mathscr{L}'} \xrightarrow{-\tau'} \mathscr{B}_{\mathscr{L}'}$. Given that we have established Lemma~\ref{torsors over group stack} and Lemma~\ref{action over phitchin}, the proof that these two morphisms are inverse of each other is almost the same as the proof of \cite[Proposition 3.9 and Theorem 3.12]{Nonabelian Hodge prime char} based on the local study in \S\ref{sect local}. The situation for $\mathcal{X}^{\rm ext}$ is similar.
\end{proof}

\subsection{Structure of $\mathcal{X}$}\label{subsect structure}
We will give a description of $\mathcal{X}$ in terms of $X'$. It turns out that the structure of $\mathcal{X}$ depends on the residue of the logarithmic connection.
\begin{lem}
Let $(E,\nabla,\Psi)$ be an element in $\mathcal{X}$. For each point $x \in D$, the residue of $\nabla$ at $x$ is a rational semisimple element.
\end{lem}

\begin{proof}
Since $\nabla$ has zero $p$-curvature, Lemma \ref{normal form for connection with zero p curvature} implies the result.
\end{proof}

Let $\tau \in\mathfrak{t}$ be a rational semisimple element, which can be considered as a representative of the corresponding conjugacy class. Any such element $\tau$ defines a natural weight $\theta_\tau$ parahoric subgroup $G'_{\tau}$ of $G$ over $X'$ as we discussed in \S\ref{sect local}. Let $(\theta_\tau)_{*}\mathcal{O}_{X}(x)$ be the corresponding $T$-torsor. Since $\mathcal{O}_{X}(x)$ is equipped with a natural connection with a pole at $x$, the $T$-torsor $(\theta_\tau)_{*}\mathcal{O}_{X}(x)$ is also equipped with a logarithmic connection $\nabla_\tau$ such that the residue is $\tau$. Let $G'_{\tau}$ be the group of automorphisms of the tame $G$-local system $((\theta_\tau)_{*}\mathcal{O}_{X}(x),\nabla_\tau)$. Lemma~\ref{parahoric of rational semisimple} implies that $G'_{\tau}$ can be identified with a parahoric group scheme over $X'$ such that the Levi factor is equal to the centralizer of $\tau$ in $G$.

Now we take a collection of rational semisimple elements ${\boldsymbol\tau}:=\{\tau_x, x \in D\}$. Let ${\rm Locsys}^{tame}_{G,{\boldsymbol\tau}} \subseteq {\rm Locsys}^{tame}_{G}$ be the stack classifying tame $G$-local systems $(E,\nabla)$ on $X$ such that
\begin{itemize}
\item $\nabla$ is of zero $p$-curvature;
\item the residue of $\nabla$ at $x$ is conjugate to $\tau_x$ for each point $x \in D$.
\end{itemize}
Denote by $G'_{\boldsymbol\tau}$ the corresponding parahoric group scheme over $X'$.

\begin{lem}\label{parahoric bundle}
The stack ${\rm Locsys}^{tame}_{G,{\boldsymbol\tau}}$ is equivalent to the stack ${\rm Bun}_{G'_{\boldsymbol\tau}}$ of $G'_{\boldsymbol\tau}$-torsors over $X'$.
\end{lem}

\begin{proof}
It is enough to prove this lemma for a single point $D=\{x\}$, and let $\tau$ be the given rational semisimple element at $x$. As we stated above, the $G$-torsor $(\theta_\tau)_{*}O_{X}(x)$ has a natural logarithmic connection $\nabla_\tau$ such that the residue is equal to $\tau$ (under conjugation) and the $p$-curvature is zero. We will prove that for any $(E,\nabla) \in {\rm Locsys}^{tame}_{G,\tau}$, we have that $(E,\nabla)$ is locally isomorphic to $((\theta_\tau)_{*}O_{X}(x),\nabla_\tau)$ as tame local systems on $X$, which will imply this lemma.

Let ${\rm Iso}\left( ((\theta_\tau)_*O_{X}(x),\nabla_\tau), (E,\nabla) \right)$ be the group of isomorphisms between $((\theta_\tau)_*O_{X}(x),\nabla_\tau)$ and $(E,\nabla)$. The isomorphism group ${\rm Iso}$ is an affine scheme over $X'$. In fact, ${\rm Iso}$ is smooth and maps surjectively over $X'$. Away from $x$, this follows from Cartier descent. Lemma~\ref{normal form for connection with zero p curvature} implies that $x'$ is in the image of the morphism ${\rm Iso} \rightarrow X'$. The smoothness of ${\rm Iso}$ over $x'$ is a consequence of Lemma~\ref{standard form over artinian rings}, Lemma~\ref{normal form for connection with zero p curvature} and the following Lemma~\ref{criterion for smoothness over a curve}.
\end{proof}

\begin{lem}\label{criterion for smoothness over a curve}
Let $U$ be a finite type affine scheme over a curve $X$. Given $x \in X$, let $z$ be a local coordinate at $x$. Then $U$ is smooth over $x$ if the functor
\begin{equation}
H(R)={\rm Hom}_{X}({\rm Spec}(R[[z]]),Y)
\end{equation}
on Artinian local algebras over $k$ is formally smooth.
\end{lem}

\begin{proof}
Let $R$, $R'$ be two Artinian local algebras over $k$ such that we have a closed embedding $\Spec(R)\rightarrow\Spec(R')$. Given a diagram
$$
\xymatrix{
 & & Y\ar[d]\\
\Spec(R) \ar[r] \ar[rru]^{f} & \Spec(R') \ar[r] \ar@{-->}[ru]^{f'} & X \ ,
}
$$
we need to show that $f$ lifts to $f'$. Composing $f$ with ${\rm Spec}(R[[z]])\rightarrow \Spec(R)$, one can view $f$ as an element in $H(R)$. If $H$ is formally smooth, then $f$ lifts to $f'\in H(R')$. Now the morphism $\Spec(R')\rightarrow X$ endows $\Spec(R')$ with the structure as a subscheme of ${\rm Spec}(R'[[z]])$. Therefore, one finds a lift of $f$ to $f'$
\end{proof}

Let $\mathcal{X}_{\boldsymbol\tau}$ be the substack of $\mathcal{X}$, which parametrizes triples $(E,\nabla,\Psi)$ such that the residue of $\nabla$ at $x$ is conjugate to $\tau_x$.

\begin{prop}\label{structure of X}
We have
\begin{align*}
\mathcal{X}_{\boldsymbol\tau} \cong {\rm Higgs}^{tame}_{G'_{\boldsymbol\tau}},
\end{align*}
and the stack $\mathcal{X}_{\boldsymbol\tau}$ is an open substack of $\mathcal{X}$. Hence $\mathcal{X}$ is the disjoint union of $\mathcal{X}_{\boldsymbol\tau}$, where the union ranges over the all rational semisimple conjugacy classes.
\end{prop}

\begin{proof}
Let $(E,\nabla,\Psi)\in\mathcal{X}_{\boldsymbol\tau}$ be an element. By Lemma~\ref{parahoric bundle}, the pair $(E,\nabla)$ determines a unique $G'_{\boldsymbol\tau}$-torsor $E'$. Thus, we only need to show that if $\Psi \in H^0(X,{\rm Ad}(E) \otimes Fr^* \mathscr{L})$ is a horizontal section, then $\Psi\in{\rm Ad}(E')\otimes \mathscr{L}'$. Using Cartier descent again, one only needs to prove this statement over the formal disc at $x$, so one may assume that the connection is of the form $d+\tau_x\frac{dz}{z}$. An element $A\frac{dz}{z}\in\mathfrak{g}(O_{x})\frac{dz}{z}$ is horizontal if and only if $dA+[A,\tau_x]\frac{dz}{z}=0$. By the definition of $G'_{\tau_x}$, this implies that $A\in \mathfrak{g}'_\tau(\mathcal{O})$. This finishes the proof for the isomorphism. The fact that $\mathcal{X}_{\boldsymbol\tau}$ is an open substack follows from the fact that the condition that the residue is fixed is stable under deformations by Lemma \ref{normal form for connection with zero p curvature}.
\end{proof}

\subsection{Tame Parahoric Nonabelian Hodge Correspondence}
Now we consider the more general case of tame parahoric local systems. Let $\mathcal{P}_{\boldsymbol\theta}$ be the parahoric group scheme on $X$ corresponding to a given collection of tame weights $\boldsymbol\theta$. The first is the analogue of Theorem~\ref{first structure theorem for G}. Let $\mathcal{X}_{\mathcal{P}_{\boldsymbol\theta}}$ be the stack parameterizes triples $(E,\nabla, \Psi)$ such that
\begin{itemize}
\item $(E,\nabla)$ is a tame $\mathcal{P}_{\boldsymbol\theta}$-local system with zero $p$-curvature,
\item $\Psi$ is a horizontal section in ${\rm Ad}(E)\otimes Fr^*\mathscr{L}'$.
\end{itemize}
\begin{thm}\label{parahoric non hod corr}
There exists a canonical isomorphism of stacks over $B_{\mathscr{L}'}$:
\begin{align*}
\mathcal{H}\times^{\mathcal{A}}\mathcal{X}_{\mathcal{P}_{\boldsymbol\theta}}\rightarrow {\rm Locsys}^{tame}_{\mathcal{P}_{\boldsymbol\theta}}.
\end{align*}
\end{thm}

\begin{proof}
The action
\begin{align*}
{\rm Locsys}^{tame}_{J^p} \times {\rm Locsys}^{tame}_{\mathcal{P}_{\boldsymbol\theta}} \rightarrow   {\rm Locsys}^{tame}_{\mathcal{P}_{\boldsymbol\theta}}.
\end{align*}
given in \eqref{action on locsys} induces a morphism
\begin{align*}
    \mathcal{H} \times^{\mathcal{A}} \mathcal{X}_{\mathcal{P}_{\boldsymbol\theta}} \rightarrow {\rm Locsys}^{tame}_{\mathcal{P}_{\boldsymbol\theta}}.
\end{align*}
By the canonical section $\tau': B_{\mathscr{L}'} \rightarrow \mathscr{B}_{\mathscr{L}'}$ defined in \S\ref{subsect_tnhc}, the ``inverse" morphism is also given by
\begin{align*}
    \mathcal{H}(-\tau') \times^{\mathcal{A}} {\rm Locsys}^{tame}_G \rightarrow \mathcal{X}.
\end{align*}
Then, based on the local study in \S\ref{sect local}, the same argument as in \cite[Proposition 3.9 and Theorem 3.12]{Nonabelian Hodge prime char} finishes the proof of this theorem.
\end{proof}

Similar to \S\ref{subsect structure}, we will discuss the local structure of $\mathcal{X}_{\mathcal{P}_{\boldsymbol\theta}}$ in terms of $X'$. Let $Y$ be a covering of $X$ determined by the tame weights $\boldsymbol\theta$. Let $d$ be the degree of the covering, and denote by $\Gamma$ the Galois group. We identify logahoric $\mathcal{P}_{\boldsymbol\theta}$-Higgs bundles (resp. tame $\mathcal{P}_{\boldsymbol\theta}$-local systems) on $X$ as $(\Gamma,G)$-Higgs bundles (resp. tame $(\Gamma,G)$-local systems) on $Y$ with local type $\boldsymbol\rho$, where $\boldsymbol\rho$ is determined by $\boldsymbol\theta$ (see \S\ref{sect corresp}). We also fix a collection ${\boldsymbol \tau}=\{\tau_x \text{ }| \text{ }x \in D\}$ of rational semisimple elements in $\mathfrak{t}$. Let ${\rm Locsys}^{tame}_{\mathcal{P}_{\boldsymbol\theta}, {\boldsymbol\tau}}(X)$ be the stack parameterizes tame $\mathcal{P}_{\boldsymbol\theta}$-local systems over $X$ with zero $p$-curvature and that the residue of $\nabla$ at $x \in D$ has rational semisimple part $\tau_x$.

\begin{lem}
${\rm Locsys}^{tame}_{\mathcal{P}_{\boldsymbol\theta}, {\boldsymbol\tau}}(X)$ is isomorphic to ${\rm Bun}_{G'_{ d({\boldsymbol\theta}+{\boldsymbol\tau})} }(X')$.
\end{lem}

\begin{proof}
We go with the following diagram to prove this lemma.
\begin{center}
\begin{tikzcd}
{\rm Locsys}^{tame}_{\mathcal{P}_{\boldsymbol\theta}, {\boldsymbol\tau}}(X)  \arrow[d,dotted] \arrow[rr,"\text{Proposition~\ref{parahoric connection equivariant}}"]  & & {\rm Locsys}^{tame}_{G,d({\boldsymbol\theta} + {\boldsymbol\tau})}([Y/\Gamma]) \arrow[d, "\text{Lemma~\ref{parahoric bundle}}"] \\
{\rm Bun}_{G'_{{\boldsymbol\theta}+{\boldsymbol\tau}}}(X') \arrow[rr, "\text{Theorem~\ref{equi for torsors}}"] & & {\rm Bun}_{G'_{d({\boldsymbol\theta}+{\boldsymbol\tau}) } }([Y'/G])
\end{tikzcd}
\end{center}
In the proof of Theorem~\ref{classification for parahoric connection}, we show that
\begin{align*}
    {\rm Locsys}^{tame}_{\mathcal{P}_\theta,\tau}(\mathbb{D}_x ) \cong {\rm Locsys}^{tame}_{G, d(\theta+\tau)}([\mathbb{D}_y /\Gamma]).
\end{align*}
By Proposition~\ref{parahoric connection equivariant}, this local picture can be naturally generalized as follows:
\begin{align*}
{\rm Locsys}^{tame}_{\mathcal{P}_{\boldsymbol\theta}, {\boldsymbol\tau}}(X) \cong {\rm Locsys}^{tame}_{G,d({\boldsymbol\theta} + {\boldsymbol\tau})}([Y/\Gamma]).
\end{align*}
Lemma~\ref{parahoric bundle} gives the isomorphism ${\rm Locsys}^{tame}_{G,d({\boldsymbol\theta} + {\boldsymbol\tau})}([Y/\Gamma]) \cong {\rm Bun}_{G'_{d({\boldsymbol\theta}+{\boldsymbol\tau}) } }([Y'/G])$. Note that although the lemma is proved for $G$-bundles, the result can be naturally generalized to a $\Gamma$-equivariant version. Finally, by Theorem~\ref{equi for torsors}, we have the desired isomorphism.
\end{proof}

Let $\mathcal{X}_{\mathcal{P}_{\boldsymbol\theta}, {\boldsymbol\tau}}$ be the substack of $\mathcal{X}_{\mathcal{P}_{\boldsymbol\theta}}$ such that $(E,\nabla,\Psi) \in \mathcal{X}_{\mathcal{P}_{\boldsymbol\theta}, {\boldsymbol\tau}}$ if the residue of the logarithmic connection $\nabla$ at $x \in D$ is conjugate to $\tau_x$. Sometimes we use the notation $\mathcal{X}_{\mathcal{P}_{\boldsymbol\theta}, {\boldsymbol\tau}}(X)$ to emphasize that it is defined over $X$. Following the diagram below,
\begin{center}
\begin{tikzcd}
\mathcal{X}_{\mathcal{P}_{\boldsymbol\theta}, {\boldsymbol\tau}}(X)  \arrow[d,dotted, "\text{Proposition~\ref{structure of parahoric X}}"] \arrow[rr,"\text{Proposition~\ref{parahoric connection equivariant}}"]  & & \mathcal{X}_{G, d({\boldsymbol\theta}+{\boldsymbol\tau})}([Y/\Gamma]) \arrow[d, "\text{Proposition~\ref{structure of X}}"] \\
{\rm Higgs}^{tame}_{G'_{{\boldsymbol\theta}+{\boldsymbol\tau}}}(X')  & & {\rm Higgs}^{tame}_{G'_{d({\boldsymbol\theta}+{\boldsymbol\tau})}}([Y'/\Gamma]) \arrow[ll, "\text{Theorem~\ref{equivalence for higgs bundles}}"]
\end{tikzcd}
\end{center}
we apply similar arguments as in Proposition~\ref{structure of X} and give a description of the structure of $\mathcal{X}_{\mathcal{P}_{\boldsymbol\theta}, {\boldsymbol\tau}}$.
\begin{prop}\label{structure of parahoric X}
We have
\begin{align*}
\mathcal{X}_{\mathcal{P}_{\boldsymbol\theta}, {\boldsymbol\tau}} \cong {\rm Higgs}^{tame}_{G'_{{\boldsymbol\theta}+{\boldsymbol\tau}}}(X'),
\end{align*}
and the stack $\mathcal{X}_{\mathcal{P}_{\boldsymbol\theta}}$ is the disjoint union of all $\mathcal{X}_{\mathcal{P}_{\boldsymbol\theta}, {\boldsymbol\tau}}$, where ${\boldsymbol\tau}$ ranges over a set of representatives of rational semisimple conjugacy class of $\mathfrak{g}$ under the adjoint action of $Z_{G}(\zeta^{d\theta})$.
\end{prop}

\bigskip
\noindent\small{\textsc{Department of Mathematics, University of Illinois at Urbana-Champaign}\\
1409 W. Green St, Urbana, IL 61801, USA}\\
\emph{E-mail address}: \texttt{lmaowis@gmail.com}

\bigskip
\noindent\small{\textsc{Department of Mathematics, South China University of Technology}\\
381 Wushan Rd, Tianhe Qu, Guangzhou, Guangdong, China}\\
\emph{E-mail address}:  \texttt{hsun71275@scut.edu.cn}

\end{document}